\documentclass[12pt]{amsart}
\usepackage{amsmath}
\usepackage{amssymb}
\usepackage{amsthm}

\input xy
\xyoption{matrix}
\xyoption{arrow}
\xyoption{curve}
\xyoption{2cell}
\UseTwocells

\theoremstyle{plain}
\newtheorem{theorem}{Theorem}[section]
\newtheorem{corollary}[theorem]{Corollary}
\newtheorem{proposition}[theorem]{Proposition}
\newtheorem{lemma}[theorem]{Lemma}

\theoremstyle{definition}
\newtheorem{definition}[theorem]{Definition}
\newtheorem{construction}[theorem]{Construction}
\newtheorem{notation}[theorem]{Notation}
\newtheorem{example}[theorem]{Example}

\newcommand{\ACat}{{\mathbf{C}}}
\def\b#1{\mathbf{#1}}

\newcommand{\B}{{\mathcal{B}}}
\newcommand{\Bilin}{{\mathbf{Bilin}}}
\newcommand{\blank}{\underline{\phantom{m}}}
\newcommand{\br}[1]{\langle#1\rangle}
\newcommand{\C}{{\mathcal{C}}}
\newcommand{\CC}{{\mathbf{C}}}
\newcommand{\Cat}{{\mathbf{Cat}}}
\newcommand{\Catstar}{\Cat_*}

\newcommand{\concat}{\odot}

\def\d{\displaystyle}
\newcommand{\D}{{\mathcal{D}}}
\newcommand{\E}{{\mathcal{E}}}
\newcommand{\e}{\varepsilon}
\newcommand{\Estar}{{E^{*}}}
\newcommand{\F}{{\mathcal{F}}}
\newcommand{\G}{{\mathcal{G}}}
\newcommand{\GC}{\Gstar{\textnormal{-}}\ACat}

\newcommand{\Gstar}{{\mathcal{G}_{*}}}
\newcommand{\GstarCat}{\Gstar{\textnormal{-}}\Cat}
\newcommand{\id}{{\mathop{\textnormal{id}}\nolimits}}
\long\def\ignore#1\endignore{}
\newcommand{\Inj}{{\mathbf{Inj}}}
\newcommand{\Jhat}{\hat{J}}
\newcommand{\klin}{\P_k}
\newcommand{\laxstar}{lax$_*$}
\newcommand{\multprod}{\Gamma}
\newcommand{\Hom}{{\mathbf{Hom}}}
\newcommand{\M}{{\mathbf{M}}}
\newcommand{\Mgrph}{{\mathbf{Mgrph}}}
\newcommand{\MM}{\mathbb{M}}
\newcommand{\Mult}{{\mathbf{Mult}}}
\newcommand{\Multstar}{\Mult_*}

\newcommand{\Nmult}{{\mathbf{Nmult}}}
\newcommand{\Ob}{{\mathop{\textnormal{Ob}}}}
\newcommand{\obj}[1]{{\mathop{\textnormal{Ob}}}\left(#1\right)}
\newcommand{\op}{{\textnormal{op}}}
\def\P{{\mathbf{P}}}
\newcommand{\PP}{{\mathbb{P}}}

\newcommand{\Set}{{\mathbf{Set}}}
\newcommand{\Strict}{{\mathbf{Strict}}}
\newcommand{\subst}[1]{\lceil_{\!#1}}
\let\sma\wedge
\def\u#1{\underline{#1}}
\newcommand{\Z}{{\mathbb{Z}}}

\usepackage{color}

\begin{document}

\title{Permutative categories, multicategories, and algebraic $K$-theory}

\author{A.\ D.\ Elmendorf}

\address{Department of Mathematics\\
Purdue University Calumet\\
Hammond, IN 46323}

\email{aelmendo@calumet.purdue.edu}

\author{M.\ A.\ Mandell}
\thanks{The second author was supported in part by NSF grant
DMS-050469}

\address{Department of Mathematics\\ Indiana University\\
Bloomington, IN 47405}

\email{mmandell@indiana.edu}

\date{September 1, 2009}

\begin{abstract}
We show that the $K$-theory construction of \cite{RMA}, which
preserves multiplicative structure, extends to a symmetric
monoidal closed bicomplete source category, with the
multiplicative structure still preserved.  The source category of
\cite{RMA}, whose objects are permutative categories, maps fully
and faithfully to the new source category, whose objects are
(based) multicategories.
\end{abstract}

\maketitle


\section{Introduction}
\label{Intro}

In \cite{RMA}, we introduced a $K$-theory functor from permutative
categories to symmetric spectra, equivalent to previous
definitions, which also preserves multiplicative structure. The
multiplicative structure on the category of permutative categories
was captured by a \emph{multicategory} structure, which is a
simultaneous generalization of the concepts of operad and
symmetric monoidal category.  Since symmetric spectra support a
symmetric monoidal multiplicative structure, they automatically
form a multicategory, and it is this structure that our $K$-theory
functor preserves.

Part of the added flexibility that multicategories offer is that a
full subcategory of a multicategory inherits a multicategory
structure.  In particular, a full subcategory of a symmetric
monoidal category inherits a multicategory structure, although it
will no longer be monoidal unless it is closed under the monoidal
product.  In addition, the maps of multicategories, or
\emph{multifunctors}, between symmetric monoidal categories are
simply the lax symmetric monoidal functors.  Since the $K$-theory
map produced in \cite{RMA} is a multifunctor from a multicategory
(the permutative categories) to a symmetric monoidal category (the
symmetric spectra), it is a natural question whether the source
multicategory can be embedded as a full subcategory of a symmetric
monoidal category, with the $K$-theory map extending to it as a
lax symmetric monoidal functor. We can also ask whether the new
source category can be made bicomplete, to make it a convenient
place to do homotopy theory. We answer both these questions in the
affirmative: this is the content of Theorems
\ref{emb} and \ref{smcb}, with the objects of the larger category being,
ironically enough, multicategories!  We actually use \emph{based}
multicategories, in the sense that the objects come equipped with
a preferred map from the terminal multicategory, and the embedding takes
a permutative category to its underlying based multicategory.

\begin{theorem}\label{emb}
Let $\P$ be the multicategory of permutative categories, and let
$\Multstar$ be the symmetric monoidal category of based
multicategories.  Then the underlying based multicategory
construction gives a full and faithful multifunctor from $\P$ to
$\Multstar$ of multicategories enriched over $\Cat$.
\end{theorem}

The category $\Multstar$ cures many of the defects of the
multicategory of permutative categories: in addition to being
symmetric monoidal, it is closed, complete, and cocomplete. We
will derive these properties from the same ones for unbased
multicategories.

\begin{theorem}\label{smcb}
The categories $\Mult$ of unbased multicategories and $\Multstar$
of based multicategories are both symmetric monoidal, closed, and
bicomplete.
\end{theorem}

Our final main result extends the $K$-theory map of \cite{RMA} to
a lax symmetric monoidal map from the symmetric monoidal category
of pointed multicategories, and here too there is an improvement.
The $K$-theory map of \cite{RMA} is actually a composite: the
first piece is a multifunctor from permutative categories to what
we call $\Gstar$-categories, which form a symmetric monoidal
category $\GstarCat$.  Then the second piece is a lax symmetric
monoidal functor from $\GstarCat$ to symmetric spectra.  Our
extension result produces a lax symmetric monoidal functor $\Jhat$
from pointed multicategories to $\Gstar$-categories.  In addition,
we identify $\GstarCat$ as a category of functors from an index
category $\Gstar$ to $\Cat$ and $\Jhat$ as a representable
functor.

\begin{theorem}\label{extn}
There is a lax symmetric monoidal functor $\Jhat$ from $\Multstar$
to $\GstarCat$ such that the composite with the full and faithful
functor of Theorem \ref{emb} is naturally isomorphic to the
multifunctor $J$ constructed in \cite{RMA}.  Moreover, $\Jhat$ is
the representable functor $\Multstar(\Estar,\blank)$, where
$\Estar\colon \Gstar^{\op}\to \Multstar$.
\end{theorem}

The extension we seek is then the composite of this representable
extension with the lax symmetric monoidal piece from \cite{RMA}.  In
summary, we can speak meaningfully of the $K$-theory of a
(pointed) multicategory, and the $K$-theory of a permutative
category depends only on its underlying pointed multicategory. All
the multiplicative structure captured by the $K$-theory map of
\cite{RMA} also depends only on underlying based multicategories.

The paper is organized as follows.  In section \ref{mult}, we
describe the multicategory structure on the category of
multicategories; this is what will underlie the symmetric monoidal
structure we will describe later.  We also describe the enrichment
present, which follows easily from the internal hom construction
for multicategories.  We then extend all these constructions to
the based context.  Section \ref{perm} discusses the various
categories of permutative categories of interest to us, and
reviews the multicategory structure on permutative categories from
\cite{RMA}. We will then have sufficient tools on hand to give the
proof of Theorem \ref{emb}.  Section
\ref{complete} is devoted to the deeper structure of the category
of multicategories that allows us to show that it is symmetric
monoidal closed, complete, and cocomplete, as is its pointed
analogue.  Section
\ref{Kmult} then introduces the representing object for our
lax symmetric monoidal functor $\Jhat$, and section
\ref{extend} concludes the paper by showing our represented
functor is consistent with the one defined in \cite{RMA} when
restricted to permutative categories.

It is a pleasure to acknowledge the anonymous referees of
\cite{RMA} for a very interesting comment about partial
permutative categories.  We have, however, chosen the more drastic
route all the way to multicategories in this paper, with what we
hope are satisfying results.  We also thank the referee of this
paper for a careful reading and thorough comments that have
undoubtedly improved the paper.

\section{The Multicategory of Multicategories}
\label{mult}

The basic idea of a multicategory is very simple.  Like a
category, it has objects, but the essential difference is that the
source of a morphism is a string of objects of a specified length
(including length 0), rather than a single object. The target
remains a single object.  Consequently, to compose one must
consider strings of strings, which are then concatenated to obtain
the source of the composite.  As with operads, which are simply
multicategories with a single object, there are two flavors of
multicategory: with or without permutations.
The flavor without permutations was the original one introduced by
Lambek~\cite{Lambek}.
We will be concerned
almost exclusively with the flavor \emph{with} permutations and so
we will call these simply ``multicategories'' rather than
``symmetric multicategories'' as in, for example, \cite[A.2]{L}.  The
technicalities of
the definition are now fairly straightforward, and are as follows.

\begin{definition}\label{MultDef}  A \emph{multicategory} $\M$
consists of the following:
\newcounter{multdef}
\begin{list}{(\arabic{multdef})}{\usecounter{multdef}}
\item  A collection of objects, which may form a proper class,
\item For each $k\ge
0$, $k$-tuple of objects $(a_1,\ldots,a_k)$ (the ``source'') and
single object $b$ (the ``target''), a set $\M_k(a_1,\ldots,a_k;b)$
(the ``$k$-morphisms''),
\item
A right action of $\Sigma_k$ on the collection of all
$k$-morphisms, where for $\sigma\in\Sigma_k$,
$$\sigma^*:\M_k(a_1,\ldots,a_k;b)
\to\M_k(a_{\sigma(1)},\ldots,a_{\sigma(k)};b),
$$
\item
A distinguished ``unit'' element $1_a\in\M_1(a;a)$ for each object
$a$, and
\item A ``composition''
\begin{gather}
\!\!\!\!\!\!\multprod:\;\M_n(b_1,\dotsc,b_n;c)\times
\M_{k_1}(a_{11},\dotsc,a_{1k_1};b_1)\times
\cdots\times\M_{k_n}(a_{n1},\dotsc,a_{nk_n};b_n)\notag\\
\relax\longrightarrow\M_{k_1+\cdots+k_n}(a_{11},\dotsc,a_{nk_n};c),
\notag
\end{gather}
\end{list}
all subject to the identities for an operad listed on pages 1--2
in
\cite{G}, which still make perfect sense in this context.
For greater detail, we refer the reader to \cite{RMA}, Definition
2.1. A \emph{multifunctor} is a structure preserving map of
multicategories.
\end{definition}

As with categories, if the objects of a multicategory form a set,
we call it \emph{small}; otherwise it is \emph{large}.  We obtain
the category $\Mult$ whose objects are all small multicategories.
We note that restricting attention to $1$-morphisms gives an
underlying category for each multicategory.

We may also ask that the $k$-morphisms $\M_k(a_1,\dots,a_k;b)$
take values in a symmetric monoidal category other than sets,
giving us the concept of an \emph{enriched} multicategory.  The
enriched multicategories of interest to us are all large, and
generally enriched over $\Cat$; this will be the case for $\Mult$,
in particular.

A basic observation is that if we let $\M$ be a full subcategory
of a symmetric monoidal category $(\C,\oplus,0)$, then $\M$
becomes a multicategory by defining
$$\M_k(a_1,\dots,a_k;b):=\C(a_1\oplus\dots\oplus a_k,b)
$$
(for a fixed choice of association).  The same observation also
holds if $\M$ is merely the source of a full and faithful functor
to $\C$. In particular, $0$-morphisms in $\M$ are given by
morphisms in $\C$ out of the monoidal unit $0$. As a consequence,
every symmetric monoidal category has an underlying multicategory,
and it is an interesting exercise to check that the maps of underlying
multicategories between two symmetric
monoidal categories are the lax symmetric monoidal maps (see Section
\ref{perm}).

We begin our description of the additional structure on $\Mult$ by
observing that the multifunctors between two multicategories are
themselves the objects of a multicategory.  These internal Hom
multicategories will give $\Mult$ the
structure of a closed category. It is crucial for the construction that
we use multicategories \emph{with} permutations; this construction
does not extend to non-symmetric multicategories.  For
notational convenience, we will often write lists such as
$(c_1,\dots,c_k)$ as $\br{c_i}_{i=1}^k$, or even just $\br{c_i}$
when the limits are clear.

\begin{definition}\label{Homs} We define an internal Hom
object in $\Mult$ as follows.  Given small multicategories $M$ and
$N$, we define $\Hom(M,N)$ to be a multicategory with objects the
multifunctors from $M$ to $N$.  Given a source $k$-tuple
$(f_1,\ldots,f_k)$ of multifunctors and a target multifunctor $g$,
we define a $k${\sl-natural transformation} from
$(f_1,\ldots,f_k)$ to $g$ to be a function $\xi$ that assigns to
each object $a$ of $M$ a $k$-morphism $\xi_a:(f_1a,\ldots,f_ka)\to
ga$ of $N$, such that for any $m$-morphism
$\phi:(a_1,\ldots,a_m)\to b$ in $M$, the following diagram
commutes:
$$\xymatrix{
\br{\br{f_ja_i}_{j=1}^k}_{i=1}^m\ar[r]^-{\br{\xi_{a_i}}}
\ar[d]_-{\cong}
&\br{ga_i}_{i=1}^m\ar[dd]^-{g\phi}
\\ \br{\br{f_ja_i}_{i=1}^m}_{j=1}^k \ar[d]_-{\br{f_j\phi}}
\\ \br{f_jb}_{j=1}^k \ar[r]_-{\xi_b}
&gb.}
$$
Here the unlabelled isomorphism is the permutation that reverses
the priority of the indices $i$ and $j$, i.e., it shuffles $m$
blocks of $k$ entries each into $k$ blocks of $m$ entries each.
The $k$-natural transformations are then the $k$-morphisms in the
multicategory $\Hom(M,N)$.  The composition is
induced from the composition in $N$, as are the actions by the
symmetric groups.  The reader can now verify that the axioms for a
multicategory are satisfied.
\end{definition}

We wish to define a multicategory structure on $\Mult$, and the
crucial step is the definition of $2$-morphisms, which we call
\emph{bilinear maps}.  Given multicategories $M,N$ and $P$, the
idea of a bilinear map $f:(M,N)\to P$ is a map that is
multifunctorial in each variable separately, such that the
multifunctors in each variable are ``multinatural'' in a way that
mirrors our definition of Hom multicategories.  The precise
definition follows.

\begin{definition}\label{bilin} Let $M$, $N$, and $P$ be
multicategories.  A {\emph{bilinear map}} $f:(M,N)\to P$ consists
of:
\begin{enumerate}
\item
A function $f:\obj M\times\obj N\to\obj P$,
\item
For each $m$-morphism $\phi:(a_1,\ldots,a_m)\to a'$ of $M$ and
object $b$ of $N$, an $m$-morphism
$f(\phi,b):(f(a_1,b),\ldots,f(a_m,b))\to f(a',b)$ of $P$,
\item
For each $n$-morphism $\psi:(b_1,\ldots,b_n)\to b'$ of $N$ and
object $a$ of $M$, an $n$-morphism
$f(a,\psi):(f(a,b_1),\ldots,f(a,b_n))\to f(a,b')$ of $P$,
\end{enumerate}
such that
\begin{enumerate}
\item
For each object $a$ of $M$, $f(a,\blank)$ is a multifunctor from
$N$ to $P$,
\item
For each object $b$ of $N$, $f(\blank,b)$ is a multifunctor from
$M$ to $P$,
\item
Given an $m$-morphism $\phi:(a_1,\ldots,a_m)\to a'$ in $M$ and an
$n$-morphism $\psi:(b_1,\ldots,b_n)\to b'$ in $N$, the following
diagram commutes, where the unlabelled isomorphism is given by the
permutation that reverses the priority of the indices:
$$\xymatrix{
\br{\br{f(a_i,b_j)}_{i=1}^m}_{j=1}^n \ar[r]^-{\br{f(\phi,b_j)}}
\ar[d]_-{\cong}
&\br{f(a',b_j)}_{j=1}^n\ar[dd]^-{f(a',\psi)}
\\ \br{\br{f(a_i,b_j)}_{j=1}^n}_{i=1}^m \ar[d]_-{\br{f(a_i,\psi)}}
\\ \br{f(a_i,b')}_{i=1}^m\ar[r]_-{f(\phi,b')}
&f(a',b').}
$$
\end{enumerate}
We write the set of such bilinear maps as $\Bilin(M,N;P)$.
\end{definition}

It is worth remarking that a bilinear map $f:(M,N)\to P$ gives, by
restriction to $1$-morphisms, a functor $f:M\times N\to P$ on
underlying categories, but that a multifunctor $f:M\times N\to P$
is quite different from a bilinear map: a multifunctor $M\times
N\to P$ assigns a $k$-morphism in $P$ to each pair of
$k$-morphisms from $M$ and $N$, while a bilinear map assigns an
$mn$-morphism given by the common diagonal of the above diagram to
each pair consisting of an $m$-morphism in $M$ and an $n$-morphism
in $N$.

A key point is that $\Bilin(M,N;P)$ gives the objects of a
multicategory naturally isomorphic to both $\Hom(M,\Hom(N,P))$ and
$\Hom(N,\Hom(M,P))$.  It is routine to verify that the morphisms
in this multicategory are given as follows.

\begin{definition}\label{110} Suppose given bilinear maps $f_1,\dots,f_k$
and $g$ in $\Bilin(M,N;P)$.  Then a $k$-morphism
$\xi:(f_1,\dots,f_k)\to g$ is a $k$-natural transformation in each
variable: it consists of a choice of $k$-morphism
$\xi_{(a,b)}:(f_1(a,b),\dots,f_k(a,b))\to g(a,b)$ in $P$ for each
pair of objects $(a,b)$ in $\Ob(M)\times\Ob(N)$, such that for
each morphism $\phi:(a_1,\dots,a_m)\to a'$ of $M$ and each
morphism $\psi:(b_1,\dots,b_n)\to b'$ of $N$ the following pair of
diagrams commutes:
$$\xymatrix{
\br{\br{f_j(a_i,b)}_{i=1}^m}_{j=1}^k \ar[r]^-{\br{f_j(\phi,b)}}
\ar[d]_-{\cong}
&\br{f_j(a',b)}_{j=1}^k \ar[dd]^-{\xi_{(a',b)}}
\\ \br{\br{f_j(a_i,b)}_{j=1}^k}_{i=1}^m
\ar[d]^-{\br{\xi_{(a_i,b)}}}
\\ \br{g(a_i,b)}_{i=1}^m \ar[r]^-{g(\phi,b)}
&g(a',b),}
\qquad
\xymatrix{
\br{\br{f_j(a,b_i)}_{i=1}^n}_{j=1}^k \ar[r]^-{\br{f_j(a,\psi)}}
\ar[d]_-{\cong}
&\br{f_j(a,b')}_{j=1}^k \ar[dd]^-{\xi_{(a,b')}}
\\ \br{\br{f_j(a,b_i)}_{j=1}^k}_{i=1}^n
\ar[d]^-{\br{\xi_{(a,b_i)}}}
\\ \br{g(a,b_i)}_{i=1}^n \ar[r]^-{g(a,\psi)}
&g(a,b').}
$$
\end{definition}

\begin{corollary} With morphisms as in Definition
\ref{110}, $\Bilin(M,N;P)$ forms a multicategory with natural
isomorphisms of multicategories
$$\Hom(M,\Hom(N,P))\cong\Bilin(M,N;P)\cong\Hom(N,\Hom(M,P)).
$$
\end{corollary}

The definition of a $k$-linear map of multicategories for $k\ge2$
presents no further difficulties, as it is merely a map that is
multifunctorial in each variable separately, and bilinear in each
pair of variables separately.

\begin{definition}\label{klin}
Let $M_1,\dots,M_k$ and $N$ be multicategories.  A
\emph{$k$-linear map} $f:(M_1,\dots,M_k)\to N$ consists of:
\begin{enumerate}
\item A function $\Ob(M_1)\times\dots\times\Ob(M_k)\to\Ob(N)$
\item For each $m$-morphism $\phi_j:(a_{j1},\dots,a_{jm})\to a_j'$
in $M_j$ and choices of objects $b_i$ in $M_i$ for $i\ne j$, an
$m$-morphism
\begin{align}
&f(b_1,\dots,b_{j-1},\phi,b_{j+1},\dots,b_k)\notag
\\:&(f(b_1,\dots,b_{j-1},a_{j1},b_{j+1},\dots,b_k),
\dots,f(b_1,\dots,b_{j-1},a_{jm},b_{j+1},\dots,b_k))
\notag
\\ \to &f(b_1,\dots,b_{j-1},a',b_{j+1},\dots,b_k)\notag
\end{align}
\end{enumerate}
such that
\begin{enumerate}
\item
$f$ is multifunctorial in each variable separately, and
\item
$f$ is bilinear in each pair of variables separately.
\end{enumerate}
\end{definition}

Finally, we state explicitly that a $0$-morphism to a multicategory
$N$ consists of a choice of object of $N$.  With these definitions
in place, it is now possible to verify directly that $\Mult$
becomes a multicategory with the evident notion of composition;
however, the details are unnecessary, since we will show in Section
\ref{complete} that $\Mult$ actually supports the structure of a symmetric
monoidal closed category of which this is the underlying
multicategory. We remark that the Hom construction introduced here
will give the closed part of this structure.  In particular,
remembering only the $1$-morphisms in the Hom construction gives
$\Mult$ an enrichment over $\Cat$, which we use in our extension of
the $K$-theory functor in Section~\ref{Kmult}.

We turn next to the multicategory of main interest in this paper,
whose objects are \emph{based} multicategories.  We note first
that $\Mult$ has a terminal object, namely the multicategory with
one object and one $k$-morphism for each $k\ge0$.  We will denote
this terminal multicategory by the usual unilluminating $*$. Since
it has only one object, it is an operad. In fact, it is the operad
that parametrizes commutative monoids in a multicategory; i.e., a
multifunctor from $*$ to a multicategory $M$ consists of an object
in $M$ and a commutative monoid structure on that object.

\begin{definition}
A \emph{based multicategory} consists of a multicategory $M$,
together with a preferred multifunctor $*\to M$.  A \emph{based
multifunctor} is a multifunctor preserving the basepoint
structure.
\end{definition}

We obtain a category $\Multstar$ of based multicategories.  This
will be the source of our extended $K$-theory functor, and one of
the main purposes of this paper is to show that it is actually
symmetric monoidal, closed, complete, cocomplete, and that the
extended $K$-functor is lax symmetric monoidal.  We content ourselves in
this section with a description of the underlying multicategory
structure.

We say that an object or morphism in a based multicategory
\emph{comes from the basepoint} if it is in the image of the given
map from $*$.

\begin{definition}
Let $M$, $N$, and $P$ be based multicategories.  A \emph{based
bilinear map} $f:(M,N)\to P$ is a bilinear map of the underlying
(unbased) multicategories such that if input data from either
variable comes from the basepoint, the output also comes from the
basepoint.
\end{definition}

We also define a $0$-morphism to a based multicategory to be
simply a $0$-morphism to the underlying unbased multicategory,
that is, a choice of object.  (As an analogue, think of a
$0$-morphism to a pointed topological space as a based map from
the unit of the smash product, $S^0$.)  It is now straightforward
to extend the definition of a $k$-linear map of multicategories
(Definition
\ref{klin}) to the based context by attaching the word ``based''
where appropriate.  Note also that the morphism induced by a
$\phi_j$ in $M_j$ must come from the basepoint whenever  $\phi_j$
or any of the $b_i$ come from the basepoint.  It is now
straightforward, but tedious and unnecessary, to verify that we do
get a multicategory of based multicategories -- again unnecessary
because we will show that $\Multstar$ actually supports a
symmetric monoidal structure for which this is the underlying
multicategory structure.

We close this section with some remarks on enrichment.  We
will show in Section~\ref{complete} that both $\Mult$ and $\Multstar$
are symmetric
monoidal closed, and so in particular are enriched over themselves.
We obtain an enrichment over $\Cat$ from the lax symmetric
monoidal forgetful functors
$$\Multstar\to\Mult\to\Cat,
$$
where $\Cat$ is given its
Cartesian monoidal structure.  This is the enrichment of greatest
topological significance, although we will use the full
enrichment of $\Mult$ over itself to show that $\Multstar$ is
symmetric monoidal closed.

\section{Permutative Categories}
\label{perm}

In this section we introduce the categories of permutative
categories we will need, review the multicategory structure on
permutative categories from \cite{RMA}, and prove that this
multicategory of permutative categories admits a full and faithful
multifunctor to $\Multstar$.

A permutative category (also called a ``symmetric strict monoidal
category'') is a more rigid version of a symmetric
monoidal category: it is as rigid as possible without giving up
homotopical generality.  The precise definition is as follows.

\begin{definition}\label{PermCat}
A \emph{permutative category} is a category $\C$ with a functor
$\oplus\colon \C\times\C\to\C$, an object $0\in\obj\C$, and a
natural isomorphism $\gamma\colon a\oplus b\cong b\oplus a$
satisfying:
\begin{enumerate}
\item $(a\oplus b)\oplus c=a\oplus(b\oplus c)$ (strict associativity),
\item $a\oplus0=a=0\oplus a$ (strict unit), and
\item The following three diagrams must commute:
$$\xymatrix{a\oplus 0\ar[rr]^\gamma_\cong\ar[dr]_=&&0\oplus a\ar[dl]^=\\
&a,\\}
\qquad\qquad
\xymatrix{a\oplus b\ar[rr]^=\ar[dr]^\cong_\gamma&&a\oplus b\\
&b\oplus a,\ar[ur]^\cong_\gamma\\}
$$
$$\xymatrix{a\oplus b\oplus c\ar[rr]^\gamma\ar[dr]_{1\oplus\gamma}&&c\oplus
a\oplus b\\ &a\oplus c\oplus b.\ar[ur]_{\gamma\oplus 1}\\}
$$
\end{enumerate}
(Conditions (1) and (2) of course hold for both objects and morphisms.)
\end{definition}

There are several reasonable ways to define maps between
permutative categories, of which we shall need to make use of two.
First, there are the \emph{strict} maps.

\begin{definition}\label{StrictMap}
A {\emph{strict map}} of permutative categories $f:\C\to\D$ is a
functor for which $f(a\oplus b)=fa\oplus fb$ (for both objects and
morphisms), $f(0)=0$, and the
following diagram commutes:
$$\xymatrix{
f(a\oplus b)\ar[r]^-{=}\ar[d]_-{f(\gamma)} &fa\oplus
fb\ar[d]^-{\gamma}
\\f(b\oplus a)\ar[r]^-{=}
&fb\oplus fa.}
$$
\end{definition}

We obtain a category $\Strict$ of permutative categories and
strict maps.  We will make use of this category in the next
section in our proofs that $\Mult$ and $\Multstar$ are cocomplete.

Second, every permutative category $\C$ has an underlying
multicategory, and we can define a \emph{lax map} of permutative
categories to be a map of the underlying multicategories.  (This
applies to symmetric monoidal categories as well.)  A lax map
$f:\C\to\D$ can be expressed explicitly as a functor together with
natural maps
$$\eta:0\to f(0)\quad\hbox{and}\quad\lambda: f(a)\oplus f(b)\to
f(a\oplus b)
$$
subject to some coherence diagrams.  However, the fact that $f(0)$
can be distinct from 0 causes problems with basepoint control, and
consequently we will make no use of this sort of map.

Instead, the other sort of map of permutative categories we will
use exploits the fact that a permutative category actually has an
underlying \emph{based} multicategory.  The basepoint is given by
0 and all the identifications $0\oplus\dots\oplus0=0$.  We could
now define a \emph{\laxstar\ map} of permutative categories to be
a map of underlying based multicategories.  However, for
consistency with \cite{RMA}, we give the explicit description, and
the claim that this \emph{is} a map of underlying based
multicategories will follow from Theorem \ref{emb} by restriction
to $1$-morphisms. Note that these morphisms were erroneously called
``lax'' in \cite{RMA}.

\begin{definition}\label{laxstar}
Let $\C$ and $\D$ be permutative categories.  A \emph{\laxstar\
map} $f:\C\to\D$ is a functor such that $f(0)=0$, together with a
natural transformation
$$\lambda:f(a)\oplus f(b)\to f(a\oplus b),
$$
such that
\begin{enumerate}
\item
$\lambda=\id$ if either $a$ or $b$ are 0, and
\item
the following two diagrams commute:
$$\xymatrix{
f(a)\oplus f(b)\oplus f(c)\ar[r]^-{1\oplus\lambda}
\ar[d]_-{\lambda\oplus1}
&f(a)\oplus f(b\oplus c)\ar[d]^-{\lambda}
\\f(a\oplus b)\oplus f(c)\ar[r]_-{\lambda}
&f(a\oplus b\oplus c),}
\qquad
\xymatrix{
f(a)\oplus f(b)\ar[r]^-{\lambda}\ar[d]_-{\gamma}
&f(a\oplus b)\ar[d]^-{f(\gamma)}
\\f(b)\oplus f(a)\ar[r]_-{\lambda}
&f(b\oplus a).}
$$
\end{enumerate}
\end{definition}

We obtain a category $\P$ of permutative categories and \laxstar\
maps, which in fact supports the structure of a multicategory
enriched over $\Cat$.  We review the definitions from \cite{RMA}.

\begin{definition}\label{num8}
Let $\C_1,\dotsc,\C_k$ and $\D$
be small permutative categories.  We define categories
$\klin(\C_1,\dotsc,\C_k;\D)$ that provide the categories of
$k$-morphisms for the multicategory $\P$ of permutative categories
as follows. The objects of $\klin(\C_1,\dotsc,\C_k;\D)$ consist of
functors
$$f\colon \C_1\times\cdots\times\C_k\to\D,
$$
which we think of as $k$-linear maps, satisfying
$f(c_1,\dotsc,c_k)=0$ if any of the $c_i$ are 0, together with
natural transformations for $1\le i\le k$, which we think of as
distributivity maps,
$$\delta_i\colon f(c_1,\dotsc,c_i,\dotsc,c_k)\oplus
f(c_1,\dotsc,c_i',\dotsc,c_k)\to f(c_1,\dotsc,c_i\oplus
c_i',\dotsc,c_k).
$$
We conventionally suppress the variables that do not change,
writing
$$\delta_i\colon f(c_i)\oplus f(c_i')\to f(c_i\oplus c_i').
$$
We require $\delta_i=\id$ if either $c_i$ or $c_i'$ is 0, or if
any of the other $c_j$'s are 0. These natural transformations are
subject to the commutativity of the following diagrams:
$$\xymatrix{
f(c_i)\oplus f(c_i')\oplus f(c_i'')\ar[d]_-{\delta_i\oplus 1}
\ar[r]^-{1\oplus\delta_i} &f(c_i)\oplus f(c_i'\oplus c_i'')
\ar[d]^-{\delta_i}
\\
f(c_i\oplus c_i')\oplus f(c_i'')\ar[r]_-{\delta_i} &f(c_i\oplus
c_i'\oplus c_i''),\\}
\qquad
\xymatrix{
f(c_i)\oplus f(c_i')\ar[d]_-{\gamma}^-{\cong} \ar[r]^-{\delta_i}
&f(c_i\oplus c_i') \ar[d]^-{f(\gamma)}_-{\cong}
\\
f(c_i')\oplus f(c_i)\ar[r]_-{\delta_i} &f(c_i'\oplus c_i),
\\}
$$
and for $i\ne j$,
$$\xymatrix@C-80pt{&f(c_i\oplus c_i',c_j)\oplus f(c_i\oplus c_i',c_j')
\ar[ddr]^-{\delta_j}
\\
f(c_i,c_j)\oplus f(c_i',c_j)\oplus f(c_i,c_j')\oplus f(c_i',c_j')
\ar[ur]^-{\delta_i\oplus\delta_i}
\ar[dd]_-{1\oplus\gamma\oplus1}^{\cong}
\\
&&f(c_i\oplus c_i',c_j\oplus c_j').
\\
f(c_i,c_j)\oplus f(c_i,c_j')\oplus f(c_i',c_j)\oplus f(c_i',c_j')
\ar[dr]_{\delta_j\oplus\delta_j}
\\
&f(c_i,c_j\oplus c_j')\oplus f(c_i',c_j\oplus c_j')
\ar[uur]_-{\delta_i}}
$$
We explicitly define $\P_0(;\D)$ to be the category $\D$. This
completes the definition of the objects of
$\klin(\C_1,\dotsc,\C_k;\D)$.  To specify its morphisms, given two
objects $f$ and $g$, a morphism $\phi\colon f\to g$ is a natural
transformation commuting with all the $\delta_i$'s, in the sense
that all the diagrams
$$\xymatrix{f(c_i)\oplus f(c_i')\ar[d]_-{\phi\oplus\phi}
\ar[r]^-{\delta_i^f} & f(c_i\oplus c_i') \ar[d]^-{\phi}
\\
g(c_i)\oplus g(c_i') \ar[r]_-{\delta_i^g} &g(c_i\oplus c_i')
\\}$$
commute.  We also require that $\phi(c_1,\dotsc,c_k)=\id_0$
whenever any of the $c_i=0$.

In order to make the $\klin(\C_1,\dotsc,\C_k;\D)$'s the
$k$-morphisms of a multicategory, we must specify a $\Sigma_k$
action and a composition.  The $\Sigma_k$ action
$$\sigma^*f\colon \C_{\sigma(1)}\times\cdots\times\C_{\sigma(k)}
\to\D
$$
is specified by
$$\sigma^*f(c_{\sigma(1)},\dotsc,c_{\sigma(k)})
=f(c_1,\dotsc,c_k),
$$
with the structure maps $\delta_i$ inherited from $f$ (with the
appropriate permutation of the indices).  We define the
composition as follows: Given $f_j\colon
\C_{j1}\times\cdots\times\C_{jk_j}\to \D_j$ for $1\le j\le n$ and
$g\colon \D_1\times\cdots\times\D_n\to\E$, we define
$$\multprod(g;f_1,\dotsc,f_n):=g\circ(f_1\times\cdots\times f_n).$$
To specify the structure maps, suppose $k_1+\cdots+k_{j-1}<s\le
k_1+\cdots+k_j$, and let $i=s-(k_1+\cdots+k_{j-1})$.  Then
$\delta_s$ is given by the composite
$$\xymatrix{
g(f_j(c_{ji}))\oplus g(f_j(c'_{ji}))\ar[r]^-{\delta^g_j}
&g(f_j(c_{ji})\oplus f_j(c'_{ji}))\ar[r]^-{g(\delta^{f_j}_i)}
&g(f_j(c_{ji}\oplus c'_{ji})).
}$$
\end{definition}

Once we have verified that this structure maps fully and
faithfully to the multicategory $\Multstar$, it will follow that
this does define a multicategory structure on $\P$, although this
can also be done directly.  We remark that in the context of
multifunctors, ``full and faithful'' means that the multifunctor
induces a bijection (or isomorphism in the enriched context) on
the morphism sets for any particular choice of source and target.
We can now give the proof of the first of our main theorems.

\begin{proof}[Proof of Theorem \ref{emb}]
Given a permutative category $\C$, let $U\C$ be its underlying
based multicategory.  Then the claim is that $U$ extends to a full
and faithful multifunctor
$$U:\P\to\Multstar,
$$
enriched over $\Cat$.  We begin by defining $U$ on the $1$-morphisms
of $\P$, which are the \laxstar\ morphisms.  Note that by the
associativity diagram for the structure map of a \laxstar\
morphism, it induces a canonical map
$$\xymatrix{
\d\bigoplus_{i=1}^k f(a_i)\ar[r]^-{\lambda}
&f(\bigoplus_{i=1}^k a_i).}
$$
Given such a \laxstar\ morphism, we define the induced
multifunctor on the underlying based multicategories $Uf:U\C\to
U\D$ as having the same map on objects, and given a $k$-morphism
$\phi:a_1\oplus\dots\oplus a_k\to a'$ in $U\C$, we define
$Uf(\phi)$ to be the composite
$$\xymatrix{
f(a_1)\oplus\dots\oplus f(a_k) \ar[r]^-{\lambda}
&f(a_1\oplus\dots\oplus a_k) \ar[r]^-{f(\phi)}
&fa'.}
$$
More generally, given an $n$-morphism of permutative categories
$f:(\C_1,\dots,\C_n)\to\D$, we need to specify a based $n$-linear
map $Uf:(U\C_1,\dots,U\C_n)\to U\D$.  From the map being \laxstar\
in each variable separately we get a map on underlying
multicategories that is multifunctorial in each variable
separately. From the map being identically 0 whenever any input is
0 we get the basepoint condition. The only issue remaining is
whether the pentagonal diagram relating lax morphism structure
maps generates the diagram relating the variables in a bilinear
map. Since all the conditions refer to only two variables at a
time, we can reduce to the case of only two variables, and the
diagram for a bilinear map reduces in this case to the following
one:
$$\xymatrix{
\d\bigoplus_{i=1}^m\bigoplus_{j=1}^n f(a_i,b_j)
\ar[r]^-{\lambda_2} \ar[d]^-{\cong}
&{\d\bigoplus_{i=1}^n} f(a_i,\bigoplus_{j=1}^n b_j)
\ar[dd]^-{\lambda_1}
\\ \d\bigoplus_{j=1}^n\bigoplus_{i=1}^m f(a_i,b_j)
\ar[d]^-{\lambda_1}
\\ {\d\bigoplus_{j=1}^n} f(\bigoplus_{i=1}^m a_i,b_j)
\ar[r]^-{\lambda_2}
&f(\bigoplus_{i=1}^m a_i,\bigoplus_{j=1}^n b_j).}
$$
The diagram relating the lax structure maps in a $2$-morphism of
permutative categories is precisely the case $m=n=2$ of this
diagram.  Note that reversing the roles of $m$ and $n$ does not
affect the diagram.  We now proceed by induction, first holding
$n=2$ and inducting on $m$, then reversing the roles of $m$ and
$n$ to conclude that the diagram commutes for $m=2$ and arbitrary
$n$, and finally inducting on $m$ again for a fixed, but arbitrary
$n$.  This is all accomplished by Figure
\ref{Fi:1}, which displays the desired diagram for
index limits $m+1$ and $n$ and has subdiagrams for index limits
$(m,n)$ and $(2,n)$.
\begin{figure}
$$\xymatrix@C=-70pt @R=+5pt
{\d\bigoplus_{i=1}^{m+1}\bigoplus_{j=1}^n
f(a_i,b_j)\ar[rr]^-{\lambda_2}\ar[dd]_-{=}
&&{\d\bigoplus_{i=1}^{m+1}}f(a_i,\bigoplus_{j=1}^n b_j)
\ar[dd]^-{=}
\\&{\d\bigoplus_{i=1}^m} f(a_i,\bigoplus_{j=1}^n
b_j)\oplus{\d\bigoplus_{j=1}^n} f(a_{m+1},b_j)
\ar[dr]^-{1\oplus\lambda_2} \ar[dd]_(.65){\lambda_1\oplus1}
\\ {\d\bigoplus_{i=1}^m\bigoplus_{j=1}^n}
f(a_i,b_j)\oplus{\d\bigoplus_{j=1}^n}
f(a_{m+1},b_j)\ar[ur]^-{\lambda_2\oplus1} \ar[dd]_-{\cong}
&&{\d\bigoplus_{i=1}^m} f(a_i,\bigoplus_{j=1}^n b_j)\oplus
f(a_{m+1},\bigoplus_{j=1}^n b_j)\ar[dd]^-{\lambda_1\oplus1}
\\&f(\bigoplus_{i=1}^m a_i,\bigoplus_{j=1}^n b_j)\oplus
{\d\bigoplus_{j=1}^n} f(a_{m+1},b_j)\ar[dr]^-{1\oplus\lambda_2}
\\ {\d\bigoplus_{j=1}^n\bigoplus_{i=1}^m} f(a_i,b_j)\oplus
{\d\bigoplus_{j=1}^n} f(a_{m+1},b_j)\ar[dr]^-{\lambda_1\oplus1}
\ar[dd]_-{\cong}
&&f(\bigoplus_{i=1}^m a_i,\bigoplus_{j=1}^n b_j)\oplus
f(a_{m+1},\bigoplus_{j=1}^n b_j)\ar[dddd]^-{\lambda_1}
\\&{\d\bigoplus_{j=1}^n} f(\bigoplus_{i=1}^m a_i,b_j)\oplus
{\d\bigoplus_{j=1}^n} f(a_{m+1},b_j)
\ar[uu]^(.7){\lambda_2\oplus1}
\ar[ur]_-{\lambda_2\oplus\lambda_2}
\ar[dd]^-{\cong}
\\ {\d\bigoplus_{j=1}^n}\left({\d\bigoplus_{i=1}^m} f(a_i,b_j) \oplus
f(a_{m+1},b_j)\right)\ar[dr]^-{\lambda_1}\ar[dd]_-{=}
\\&{\d\bigoplus_{j=1}^n}\left(f(\bigoplus_{i=1}^m a_i,b_j)\oplus
f(a_{m+1},b_j)\right)\ar[dd]^-{\lambda_1}
\\{\d\bigoplus_{j=1}^n\bigoplus_{i=1}^{m+1}} f(a_i,b_j)
\ar[dr]^-{\lambda_1}
&&f(\bigoplus_{i=1}^{m+1}a_i,\bigoplus_{j=1}^n b_j)
\\&{\d\bigoplus_{j=1}^n} f(\bigoplus_{i=1}^{m+1}a_i,b_j)
\ar[ur]^-{\lambda_2}
 }$$
\caption{}\label{Fi:1}
\end{figure}

We have finished describing the multifunctor structure of
$U:\P\to\Multstar$, and we leave to the reader the task of
checking the necessary preservation properties.  It remains to
show that $U$ is full and faithful. To see that it is faithful,
we observe that the objects of $U\C$ are the same as those of
$\C$, and the morphisms of $\C$ are simply the $1$-morphisms of
$U\C$.  We can therefore recover the underlying category of $\C$
from $U\C$. If we have a
\laxstar\ functor $f:\C\to\D$, we can recover the functor part of
$f$ from $Uf$ by considering only its effect on $1$-morphisms, and
we recover the lax structure map $\lambda:fa\oplus fb\to f(a\oplus
b)$ by looking at the image under $Uf$ of the $2$-morphism
$\id_{a\oplus b}\in U\C(a,b;a\oplus b)$.  Now we can apply the
same argument in each variable separately for an $n$-morphism of
$\P$, say $f:(\C_1,\dots,\C_n)\to\D$, to recover $f$ from $Uf$.
Consequently, $U$ is faithful.

To see that $U$ is full, suppose first that $g:U\C\to U\D$ is a
based multifunctor, that is, a 1-morphism in $\Multstar$.  Then
certainly $g$ determines a functor $f:\C\to\D$ by restriction to
$1$-morphisms, and $f(0)=0$ because $g$ is based. We define a lax
structure map for $f$, as above, by looking at the image in $U\D$
of the canonical $2$-morphism $\id_{a\oplus b}\in U\C(a,b;a\oplus
b)$. It is now an interesting exercise to show that these data
determine a
\laxstar\ functor $f:\C\to\D$ for which $g=Uf$.

In the general case where $g:(U\C_1,\dots,U\C_k)\to U\D$ is a
$k$-morphism in $\Multstar$, we again recover a functor
$f:\C_1\times\dots\times\C_k\to\D$ by restriction to 1-morphisms,
and the structure maps for a $k$-linear map of permutative
categories are again determined by the images of the identity maps
on $c_i\oplus c_i'$ considered as a 2-morphism in
$U\C_i(c_i,c_i';c_i\oplus c_i')$. The basepoint conditions on $f$
follow from those on $g$, and the coherence pentagon relating
structure maps is the case $m=n=2$ of the diagram in definition
\ref{bilin}.  Seeing that $f$ is \laxstar\ in each variable is the
same interesting exercise as before, so we obtain a $k$-linear map
$f$ in $\P$ such that $g=Uf$.  Therefore $U$ is full.
\end{proof}

At this point, the question of whether the underlying based
multicategory functor is an embedding of multicategories is a
purely set-theoretic one.  Our treatment above does not treat the
precise point sets involved in enough detail to answer it: we have
really only constructed the functor up to natural isomorphism.
With a more precise construction, the functor is an embedding. (We
thank the referee for pointing this out.)  Specifically, the
permutative product on the objects can be recovered from the
identification of $U\C_{2}(a,b;c)$ as a subset of the morphisms in
$\C$: the source is $a\oplus b$. This then can be extended to
recover the product on the maps using the element of
$U\C_{2}(a,b;a\oplus b)$ identified with the identity in
$\C(a\oplus b,a\oplus b)$.  On the other hand, as the following
example indicates, such an observation can be somewhat misleading.

\begin{example}\label{MacLaneCategory}
Let $A$ be a based set with six elements, including the basepoint.
Then we can give $A$ two non-isomorphic group structures, say
$G_1$ and $G_2$, both having $A$ as their underlying based set,
but with $G_1\cong\Sigma_3$ and $G_2\cong\Z/6$.  Let $EA$ be the
indiscrete category with objects the elements of $A$, so all
morphism sets have exactly one element.  Then the group structures
$G_1$ and $G_2$ give $EA$ two distinct structures as a permutative
category, since the groups of objects are not isomorphic. However,
the underlying based multicategories \emph{are} isomorphic: all
$k$-morphism sets again have exactly one element for all $k$.
(This of course generalizes to any based set $A$ with non-isomorphic
monoid structures.)  Consequently, the strict isomorphism class of
a permutative category cannot be recovered from the isomorphism
class of its underlying based multicategory.
\end{example}

\section{Colimits and Tensor Products}\label{complete}

In this section we show that $\Multstar$ has all the good formal
properties required for homotopy theory.  To be specific, we will
show that it is complete, cocomplete, and supports a symmetric
monoidal closed structure whose underlying multicategory is the
one specified in Section \ref{mult}.  Our strategy is to prove
these things first in the unbased context of $\Mult$ (without the
star), and then bootstrap from there to $\Multstar$.  This
material appears to be well-known to the experts. In particular, the
symmetric monoidal structure is a special case of the tensor product
of theories of Boardman and Vogt \cite[Chapter~2]{BV}; the closed structure
is summarized in \cite[\S5.1]{MW}.  Since we have not found a full
treatment of these properties in their entirety, we produce one here.

Completeness in $\Mult$ (and $\Multstar$) is easy: as in $\Cat$,
limits are computed on objects and morphism sets within the
category of sets.

Since the construction of the monoidal product in $\Mult$ (to
which we will refer as the \emph{tensor product}) is as a colimit,
our first major goal is the cocompleteness of $\Mult$.  We start
by establishing the analogous property for the category $\Strict$
of permutative categories and strict maps.  We will then identify
$\Mult$ as a category of coalgebras over a comonad in $\Strict$,
showing it to be cocomplete.

\begin{lemma}\label{StrictCocomplete}
The category $\Strict$ of permutative categories and strict maps
is cocomplete.
\end{lemma}

\begin{proof}
The forgetful functor from $\Strict$ to $\Cat$ is the right
adjoint in a monadic adjunction; the monad is given explicitly by
$$\PP\C:=\coprod_{k\ge0}E\Sigma_k\times_{\Sigma_k}\C^k,
$$
where $E\Sigma_k$ is the translation category on the symmetric
group $\Sigma_k$. (This was apparently first pointed out by Dunn
in
\cite{D}.) A slight variation on the argument given in the proof of
\cite{EKMM}, II.7.2 shows that this monad preserves reflexive
coequalizers, and therefore by
\cite{EKMM}, II.7.4, $\Strict$ is cocomplete.
\end{proof}

It is interesting to note that the forgetful functor from
$\Strict$ to monoids that forgets the morphisms and remembers only
the objects and their monoidal structure has a right adjoint, namely
the functor that takes a monoid $M$ to the indiscrete category $EM$
that has objects $M$, single element morphism sets, and permutative
product the monoidal product.  The forgetful functor therefore
preserves coproducts:
The monoid of objects of a coproduct of permutative categories is
the coproduct of the monoids of objects of the individual categories (in
the category of monoids). Note that these are \emph{not} commutative
monoids in general, because the commutativity isomorphism has been
forgotten.

The underlying multicategory construction gives us a forgetful
functor $G:\Strict\to\Mult$, and our next goal is to show that it
has a left adjoint.

\begin{theorem}\label{Mladj} The forgetful functor
$G:\Strict\to\Mult$ has a left adjoint.
\end{theorem}

\begin{proof}
We construct the left adjoint $F$ as follows. Let $\M$ be a
multicategory. Then $F\M$ has as its objects the free monoid on
the objects of $\M$,
$$\coprod_{k\ge0}{\obj\M}^k.
$$
Given objects $\u a=(a_1,\ldots,a_m)$ and $\u b=(b_1,\ldots,b_n)$,
a morphism in $F\M$ from $\u a$ to $\u b$ consists of a function
$\phi:\{1,\ldots,m\}\to\{1,\ldots,n\}$, together with an $n$-tuple
of morphisms $(f_1,\ldots,f_n)$, where $f_j$ is a morphism of $\M$
from $\br{a_i}_{\phi(i)=j}$ to $b_j$.  Composition of morphisms is
given by composition of the set functions on the indices, together
with induced maps using the composition on the multicategory $\M$
and permutations necessary to preserve coherence.  The permutative
structure is given by concatenation of lists. The reader can now
safely verify that this does give a left adjoint.
\end{proof}

This construction is similar to the ``categories of operators''
used by May and Thomason in \cite{MT}, but differs in that it
involves the unbased sets $\{1,\ldots,n\}$ rather than the based
sets $\{0,\ldots,n\}$. The basepoint in a category of operators
encodes projection operators $a\oplus b\to a$ or $a\oplus b\to b$,
which do not exist in a general permutative category.

Our major use of this adjunction is the following:

\begin{theorem}\label{comonadic} The adjunction
$$F:\Mult\rightleftarrows\Strict:G
$$
is comonadic, i.e., the canonical comparison functor from $\Mult$
to the category of coalgebras over the comonad $FG$ on $\Strict$
is an equivalence of categories.
\end{theorem}

\begin{proof}  We use the dual form of Beck's Theorem; see \cite{BW},
Theorem 3.14. We must show that $F$ has a right adjoint, reflects
isomorphisms, that $\Mult$ has equalizers of reflexive $F$-split
equalizer pairs, and that $F$ preserves them. We already have the
right adjoint, namely the forgetful functor $G$, and we already
know that $\Mult$ is complete, so it has all the equalizers
required. We will show that $F$ reflects isomorphisms, and that it
preserves all equalizers, not just the ones required for the
hypotheses of Beck's Theorem.

To see that $F$ reflects isomorphisms, we note that for any map
$\alpha:M\to N$ of multicategories, the diagram
$$\xymatrix@C=10pt{
FM\ar[rr]^{F\alpha}\ar[dr]
&&FN\ar[dl]
\\&F(*)}
$$
commutes, where $*$ is the terminal multicategory having one
object and one $k$-morphism for every $k$, and the unlabelled
arrows are induced by $F$ from the maps to this terminal object.
The permutative category $F(*)$ has objects $[k]=\{1,\ldots,k\}$
for $k\ge0$, ordinary functions as morphisms, and sum operation
$[k_1]\oplus[k_2]:=[k_1+k_2]$, using the obvious extension to sums
of functions. Further, the image of the unit $\eta:M\to GFM$ is
precisely the preimage of the full sub-multicategory of $GF(*)$
generated by the single object $\{1\}$. Now suppose $F\alpha$ is
an isomorphism; we wish to show that $\alpha$ must itself be an
isomorphism. Then $GF\alpha$ is an isomorphism, so we get the
commutative diagram
$$\xymatrix@C=10pt{
M\ar[rr]^-{\alpha}\ar[d]^-{\eta}
&&N\ar[d]^-{\eta}
\\GFM\ar[rr]^-{GF\alpha}_-{\cong}\ar[dr]
&&GFN\ar[dl]
\\&GF(*).}
$$
But since $M$ and $N$ are precisely the preimages of the full
sub-multicategory generated by $\{1\}$, and $GF\alpha$ is an
isomorphism, it follows that $\alpha$ must be an isomorphism as
well.

To see that $F$ preserves equalizers, we observe that equalizers
in $\Strict$ are created in $\Cat$, since $\Strict$ is monadic
over $\Cat$, and further, that equalizers in $\Cat$ are simply
computed in $\Set$ on objects and morphisms separately.  It is now
straightforward to use the definition of $F$ to see that
equalizers are preserved.  (Note, however, that $F$ does not
preserve products.)
\end{proof}

We get as an immediate corollary:

\begin{corollary}\label{cocomplete} The category $\Mult$ is
cocomplete.
\end{corollary}

\begin{proof} It is equivalent to the category of coalgebras over a
comonad on the cocomplete category $\Strict$.  See \cite{CatWork},
VI.2, exercise 2 for the dual statement.
\end{proof}

We turn next to the construction of the tensor product in $\Mult$.
The composite of a bilinear map $(M,N)\to P$ with an ordinary map
of multicategories $P\to Q$ is again a bilinear map, and the
tensor product of multicategories that we will construct is a
universal bilinear target relative to ordinary maps. This tensor
product is equivalent to the tensor product of theories
constructed in
\cite{BV}, but we give an explicit construction in terms of
multicategories, rather than theories.

\begin{theorem}\label{tensorprod} For any two multicategories
$M$ and $N$, there is a tensor product multicategory $M\otimes N$
and a universal bilinear map $(M,N)\to M\otimes N$.  This tensor
product makes $\Mult$ into a symmetric monoidal category.
\end{theorem}

In order to prove Theorem \ref{tensorprod}, we must first discuss
some other categories and adjunctions related to $\Mult$, using a
modified version of Leinster's discussion of multicategories as
generalized monoids in \cite{L}. Let $\MM$ be the free monoid
monad in $\Set$,
$$\MM A:=\coprod_{k\ge0}A^k.
$$

\begin{definition}\label{Mgraph}  An $\MM${\emph{-graph}} $X$ consists
of two sets, $X_0$ (the \emph{objects}) and $X_1$ (the
\emph{arrows}), together with two functions, the {\sl source}
$s:X_1\to\MM X_0$ and the {\sl target} $t:X_1\to X_0$.  We usually
display an $\MM$-graph as a span
$$\xymatrix{\MM X_0&X_1\ar[l]_-{s}\ar[r]^-{t}&X_0.}
$$
A {\sl map} of $\MM$-graphs $f:X\to Y$ consists of functions
$f_0:X_0\to Y_0$ and $f_1:X_1\to Y_1$ for which the obvious
diagram
$$\xymatrix{
\MM X_0\ar[d]_-{\MM f_0}
&X_1\ar[l]_-{s}\ar[r]^-{t}\ar[d]^-{f_1}
&X_0\ar[d]^-{f_0}
\\ \MM Y_0
&Y_1\ar[l]^-{s}\ar[r]_{t}
&Y_0}
$$
commutes.
\end{definition}

We get a category $\Mgrph$ of $\MM$-graphs, and there is a
forgetful functor $U:\Mult\to\Mgrph$ that remembers the objects,
morphisms, sources, and targets, but forgets about the identities,
permutations, and composition. We use the following theorem in
our construction of the tensor product.

\begin{theorem}\label{Uladj} The forgetful functor
$U:\Mult\to\Mgrph$ has a left adjoint $L:\Mgrph\to\Mult$.
\end{theorem}

\begin{proof}  We proceed in two steps, using as an intermediate
stop the category $\Nmult$ of non-symmetric multicategories,
applying some observations of \cite{L} (where non-sym\-met\-ric
multicategories are called simply multicategories). Clearly the
forgetful functor $U$ factors through $\Nmult$, and we claim that
both terms in the composite have left adjoints.  The desired left
adjoint is then the composite of these two left adjoints.

First, consider the forgetful functor $U':\Nmult\to\Mgrph$.  As
observed in
\cite[2.3]{L}, $U'$ has a left
adjoint $L'$, constructed as follows.  Given an $\MM$-graph
$X$, the free non-symmetric multicategory $L'X$ is a multicategory
where the $k$-morphisms are the trees with $k$ leaves generated by
the arrows in $X$, with all nodes (including the root and the
leaves) labelled by objects of $X$. In detail, the objects of
$L'X$ are $X_0$, and the morphisms of $L'X$, called trees, are
generated recursively by the following requirements:
\begin{enumerate}
\item
For each object $a\in X_0$, there is an identity tree $1_a$.
\item
Each element $f\in X_1$ is a tree with source $sf$ and target
$tf$.
\item
Given an $f\in X_1$ with target $c$ and source $(b_1,\ldots,b_k)$,
and trees $R_1,\ldots,R_k$, not all identity trees, with the
target of $R_j$ being $b_j$ and the source being
$(a_{j1},\ldots,a_{jn_j})$, there is a tree $(f;R_1,\ldots,R_k)$
with source $(a_{11},\ldots,a_{kn_k})$ and target $c$.
\end{enumerate}

We must define a composition on this collection of trees, and we
do so by induction on the height of a tree, where we define the
height of an identity tree to be 0, the height of an element of
$X_1$ to be 1, and the height of a tree $(f;R_1,\ldots,R_k)$ to be
$1+\max_j\{\hbox{height}(R_j)\}$.  Given trees $A;B_1,\ldots,B_n$
which are composable, we define $\Gamma(A;B_1,\ldots,B_n)$ by
induction on the height of $A$.  If $A$ has height 0, then it is
$1_a$ for some object $a$, $n=1$, and $B=B_1$ has output $a$.  We
define $\Gamma(1_a;B)=B$, as required for a multicategory.
Similarly, if $B_1,\ldots,B_n$ are all identity trees, then we
require $\Gamma(A,B_1,\ldots,B_n)=A$.

If the height of $A$ is 1, then $A=f$ for some $f\in X_1$, and we
define $\Gamma(A;B_1,\ldots,B_n)=(f;B_1,\ldots,B_n)$.  For taller
trees, $A$ must be of the form $(f;A_1,\ldots,A_k)$, with $n$
partitioning into $k$ segments so the $j$'th segment of $B$'s
feeds into $A_j$.  Then we define
$$\Gamma(A;B_1,\ldots,B_n)
=(f;\Gamma(A_1;B_1,\ldots,B_{n_1}),
\ldots,\Gamma(A_k;B_{n-n_k+1},\ldots,B_n)),
$$
where the compositions on the right side are already defined,
since the heights of the $A_j$'s are all less than the height of
$A$.  It is now a routine exercise to show that the requirements
for a composition are satisfied, and that this construction
provides a left adjoint to the forgetful functor
$U':\Nmult\to\Mgrph$.

Next, we consider the forgetful functor $U'':\Mult\to\Nmult$, and
construct a left adjoint $L''$.  Given a non-symmetric
multicategory $P$, we construct $L''P$ as follows.  First, the
objects of $L''P$ are the same as the objects of $P$.  Next, given
a source string $(a_1,\ldots,a_k)$ and a target $b$, we define
$$L''P(a_1,\ldots,a_k;b)
=\coprod_{\sigma\in\Sigma_k}P(a_{\sigma^{-1}(1)},\ldots,
a_{\sigma^{-1}(k)};b),
$$
so a $k$-morphism in $L''P(a_1,\ldots,a_k;b)$ consists of an
ordered pair $(f,\sigma)$ where $\sigma\in\Sigma_k$ and $f\in
P(a_{\sigma^{-1}(1)},\ldots, a_{\sigma^{-1}(k)};b)$.  We let
$\Sigma_k$ act on the right via the natural group action on the
symmetric group coordinate.  The composition, which is forced by
the equivariance requirements for a multicategory, is given by
\begin{gather}
\Gamma((f,\sigma);(g_1,\tau_1),\ldots,(g_k,\tau_k))\notag\\
=(\Gamma(f;g_{\sigma^{-1}(1)},\ldots,g_{\sigma^{-1}(k)}),
\sigma\br{n_1,\ldots,n_k}\circ(\tau_1\oplus\cdots\oplus\tau_k)).
\notag
\end{gather}
Again, it is an exercise to show that this construction satisfies
the requirements for a multicategory, and gives a left adjoint to
the forgetful functor.  The composite $L=L''\circ L'$ therefore
gives a left adjoint for $U=U'\circ U''$.
\end{proof}

We will also need the following proposition in our construction of
the tensor product.

\begin{proposition}\label{Obadj} The set-of-objects functor
$\obj{\blank}:\Mult\to\Set$ has both a left and a right adjoint.
\end{proposition}

\begin{proof} The left adjoint assigns to a set $A$ the
multicategory $FA$ with the set $A$ as its objects and with only
identity morphisms, while the right adjoint $RA$ also has $A$ as
its objects, but with exactly one morphism for each possible
source and target.  The necessary verifications are trivial.
\end{proof}

\begin{corollary} The objects of a limit or colimit of
multicategories are computed in $\Set$.
\end{corollary}

We are now in a position to construct the tensor product and prove
that it is a universal bilinear target.

\begin{construction}\label{tensorconstruction} Let $M$ and $N$ be multicategories.
Then we can construct the coproducts of multicategories
$$\coprod_{a\in\obj{M}}N\quad\hbox{and}\quad\coprod_{b\in\obj{N}}M.
$$
Each of these coproducts has $\obj{M}\times\obj{N}$ as its set of
objects.  Further, the first of them is universal for any map that
sends the objects $\obj{M}\times\obj{N}$ to the objects of a
multicategory $P$ and which is a multifunctor in $N$; similarly,
the second is universal for maps that are multifunctors in $M$.
Any bilinear map $(M,N)\to P$ therefore induces multifunctors from
each of these coproducts to $P$ that restrict to the same map on
the common set of objects, and therefore induces a map from the
pushout
$$\xymatrix{
F(\obj{M}\times\obj{N})\ar[r]\ar[d]
&\displaystyle\coprod_{b\in\obj{N}}M\ar[d]
\\ \displaystyle\coprod_{a\in\obj{M}}N\ar[r]
&M\#N,}
$$
where the upper left corner is the free multicategory on the set
of objects $\obj{M}\times\obj{N}$.  This pushout $M\#N$ is
universal with respect to maps that are multifunctors in each
variable separately, and what remains is to make the bilinearity
diagrams of Definition \ref{bilin} commute universally.

Given one morphism from each multicategory, say an $m$-morphism
$\phi:(a_1,\ldots,a_m)\to a'$ in $M$ and an $n$-morphism
$\psi:(b_1,\ldots,b_n)\to b'$ in $N$, we define $\MM$-graphs
$X(\phi,\psi)$ and $Y(\phi,\psi)$ as follows. The objects of both
will be
$$(\{a_1,\ldots,a_m\}\times\{b_1,\ldots,b_n\})\cup\{(a',b')\},
$$
and in $X(\phi,\psi)$ there are to be precisely two arrows, both
with source $\br{\br{(a_i,b_j)}_{i=1}^m}_{j=1}^n$ and target
$(a',b')$, while in $Y(\phi,\psi)$ there is exactly one arrow with
the same source and target as the arrows in $X(\phi,\psi)$.  There
is an obvious map of $\MM$-graphs $X(\phi,\psi)\to Y(\phi,\psi)$
collapsing the two arrows of $X$ to the one arrow of $Y$.  There
is also a map of $\MM$-graphs from $X(\phi,\psi)$ to $U(M\#N)$
sending each arrow to one way around the diagram in Definition
\ref{bilin}, but without the $f$'s.  We take the adjoints of all
of these maps and form the following pushout:
$$\xymatrix{
\displaystyle\coprod_{(\phi,\psi)}LX(\phi,\psi)\ar@<1ex>[r]\ar[d]
&\mathop{M\#N}\limits_{\phantom\phi}\ar[d]
\\ \displaystyle\coprod_{(\phi,\psi)}LY(\phi,\psi)\ar@<1ex>[r]
&\mathop{M\otimes N}\limits_{\phantom\phi}.}
$$
It now follows that we have universally forced the diagrams in the
definition of a bilinear map to commute, so $M\otimes N$ is a
universal bilinear target.  As we can see by checking at each
step, the objects of $M\otimes N$ are still $\Ob(M)\times\Ob(N)$.
\end{construction}

Our next goal is the proof of the following theorem; cf. \cite[II.2.18]{BV}.

\begin{theorem}\label{tensormonoidal} The tensor product of Construction
\ref{tensorconstruction} and the internal Hom object of Definition
\ref{Homs} make $\Mult$ into a symmetric monoidal closed category.
\end{theorem}

It is clear from the construction that the tensor product is
symmetric, and it is easy to verify that it is adjoint to the
internal Hom; further, the unit is easily seen to be the
multicategory with one object and only the identity morphism on
that object.  This leaves the associativity of the tensor product
to verify, and our strategy is to enrich the Hom-tensor adjunction
and use the Yoneda Lemma.

\begin{definition} Let $S$ be a set of morphisms in a
multicategory $M$.  The multicategory $\br{S}$ {\emph{generated
by}} $S$ is the smallest sub-multicategory of $M$ that contains
all the morphisms in $S$. If $\br{S}=M$, we say that $S$ is a
{\emph{generating set of morphisms}} for $M$, or that $M$ is
{\emph{generated by}} $S$.
\end{definition}

The following proposition is clear from the construction of $L$ in
the proof of Theorem~\ref{Uladj}.

\begin{proposition}\label{genclosed}
The morphisms of $\br{S}$ consist
of the identities on the objects appearing as targets or
components of sources in $S$, together with those constructed
recursively from $S$ by means of permutations and compositions.
\end{proposition}

\begin{lemma}\label{natural} Let $A$ and $B$ be multicategories, let
$f_1,\dots,f_k,g\in\Mult(A,B)$, and suppose we have a function
$\xi$ assigning to each object $a$ of $A$ a $k$-morphism
$\xi_a:(f_1a,\dots,f_ka)\to ga$ such that the diagram of
Definition
\ref{Homs} commutes for all $\phi$ in a generating set for $A$.
Then the diagram commutes for all morphisms $\phi$ of $A$, so
$\xi$ is a $k$-natural transformation, i.e., a $k$-morphism in
$\Hom(A,B)$.
\end{lemma}

\begin{proof}  We will say that $\xi$ is natural with respect to
those morphisms for which the diagram commutes. We show that the
diagram commutes for compositions and permutations of elements
with respect to which $\xi$ is natural, so by Proposition
\ref{genclosed}, commutes for all morphisms of $A$.

First, suppose we are given composable elements with respect to
which $\xi$ is natural, say $\phi_1,\dots,\phi_n$ with
$\phi_i:(a_{i1},\dots,a_{im_i})\to b_i$ and
$\psi:(b_1,\dots,b_n)\to c$.  Then the following diagram shows
that $\xi$ is natural with respect to
$\Gamma(\psi;\phi_1,\dots,\phi_n)$:
$$\xymatrix@C+25pt{
\br{\br{\br{f_j(a_{is})}_{s=1}^{m_i}}_{i=1}^n}_{j=1}^k
\ar[r]^-{\br{\br{f_j\phi_i}_{i=1}^n}_{j=1}^k} \ar[d]_-{\cong}
&\br{\br{f_jb_i}_{i=1}^n}_{j=1}^k \ar[r]^-{\br{f_j\psi}_{j=1}^k}
\ar[d]_-{\cong}
&\br{f_jc}_{j=1}^k \ar[ddd]^-{\xi_c}
\\ \br{\br{\br{f_j(a_{is})}_{s=1}^{m_i}}_{j=1}^k}_{i=1}^n
\ar[r]^-{\br{\br{f_j\phi_i}_{j=1}^k}_{i=1}^n} \ar[d]_-{\cong}
&\br{\br{f_jb_i}_{j=1}^k}_{i=1}^n
\ar[dd]^-{\br{\xi_{b_i}}_{i=1}^n}
\\ \br{\br{\br{f_j(a_{is})}_{j=1}^k}_{s=1}^{m_i}}_{i=1}^n
\ar[d]^-{\br{\br{\xi_{a_{is}}}_{s=1}^{m_i}}_{i=1}^n}
\\ \br{\br{g(a_{is})}_{s=1}^{m_i}}_{i=1}^n
\ar[r]^-{\br{g\phi_i}_{i=1}^n}
&\br{gb_i}_{i=1}^n \ar[r]^-{g\psi}
&gc.}
$$

Now suppose given also $\sigma\in\Sigma_n$.  Then the following
diagram shows that $\xi$ is natural with respect to
$\psi\cdot\sigma$:
$$\xymatrix@C+25pt{
\br{\br{f_jb_{\sigma(i)}}_{i=1}^n}_{j=1}^k \ar[r]^-{\cong}
\ar[d]^-{\br{\sigma}}
&\br{\br{f_jb_{\sigma(i)}}_{j=1}^k}_{i=1}^n
\ar[r]^-{\br{\xi_{b_{\sigma(i)}}}_{i=1}^n} \ar[d]^-{\br{\sigma}}
&\br{gb_{\sigma(i)}}_{i=1}^n \ar[d]^-{\br{\sigma}}
\\ \br{\br{f_jb_i}_{i=1}^n}_{j=1}^k \ar[r]^-{\cong}
\ar[d]^-{\br{f_j\psi}_{j=1}^k}
&\br{\br{f_jb_i}_{j=1}^k}_{i=1}^n \ar[r]^-{\br{\xi_{b_i}}_{i=1}^n}
&\br{gb_i}_{i=1}^n \ar[d]^-{g\psi}
\\ \br{f_jc}_{j=1}^k \ar[rr]^-{\xi_c}
&&gc.}
$$

Since we were given that $\xi$ was natural with respect to
morphisms in a generating set for $A$, it now follows that it is
natural with respect to all morphisms in $A$, and therefore $\xi$
is a $k$-natural transformation.
\end{proof}

\begin{notation} Let $M$ and $N$ be multicategories, $\phi$ a
morphism of $M$ and $b$ an object of $N$.  Then we write
$\phi\otimes b$ for the morphism of $M\otimes N$ induced from
$\phi$ and $b$ by the universal bilinear map $(M,N)\to M\otimes
N$.  Similarly, we write $a\otimes\psi$ given an object $a$ of $M$
and a morphism $\psi$ of $N$.
\end{notation}

We obtain the following proposition from the universal property of the
tensor product.

\begin{proposition} The morphisms of $M\otimes N$ of the form
$a\otimes\psi$ and $\phi\otimes b$ generate the entire
multicategory $M\otimes N$.
\end{proposition}

Combining the previous proposition with Lemma~\ref{natural}, we obtain
the following proposition.

\begin{proposition}\label{tensorGen}
The $k$-morphisms of $\Hom(M\otimes N,P)$ are
precisely those functions as in Lemma~\ref{natural} that are natural
with respect to all morphisms of the form
$a\otimes \psi$ or $\phi\otimes b$.
\end{proposition}

The enriched adjunction we desire is now the following.

\begin{proposition} The adjunction
$$\Mult(M\otimes N,P)\cong\Mult(M,\Hom(N,P))
$$
enriches to a natural isomorphism of multicategories
$$\Hom(M\otimes N,P)\cong\Hom(M,\Hom(N,P)).
$$
\end{proposition}

\begin{proof}  Lemma \ref{natural} and Proposition \ref{tensorGen}
show that the isomorphism on objects
$$\Mult(M\otimes N,P)\cong\Bilin(M,N;P)
$$
also gives an isomorphism of multicategories
$$\Hom(M\otimes N,P)\cong\Bilin(M,N;P)
$$
using the morphisms on the right given in Definition \ref{110}.
However, these morphisms are precisely those giving an isomorphism
of multicategories
$$\Bilin(M,N;P)\cong\Hom(M,\Hom(N,P)),
$$
and composing these isomorphisms gives the desired enriched
adjunction.
\end{proof}

The proof that the tensor product is associative now proceeds as
follows. We have
\begin{gather}
\Mult((M\otimes N)\otimes P,Q)\cong \Mult(M\otimes
N,\Hom(P,Q))\notag
\\ \cong\Mult(M,\Hom(N,\Hom(P,Q))) \cong\Mult(M,\Hom(N\otimes
P,Q))\notag
\\ \cong\Mult(M\otimes(N\otimes P),Q).\notag
\end{gather}
The result now follows from the Yoneda Lemma. The analogous argument
with four factors proves that this associativity isomorphism satisfies
the pentagon law.  The unit diagrams are clear, and this completes the
proof that $\Mult$ is symmetric monoidal, closed, and bicomplete.

Next, we wish to establish the same properties for $\Multstar$. We
exploit the following general construction and lemma about
symmetric monoidal closed bicomplete categories.

\begin{construction}\label{smash} Let $(\CC,\otimes,\hom)$ be a symmetric
monoidal closed bicomplete category with terminal object $t$, and
let $\CC_*$ be the category of objects under $t$ in $\CC$.  For
objects $a$ and $b$ in $\CC_*$, we define their \emph{smash
product} $a\wedge b$ to be the object of $\CC_*$ given by the
following pushout in $\CC$:
$$\xymatrix{
(a\otimes t)\amalg(t\otimes b)\ar[r]\ar[d]&a\otimes b\ar[d]
\\t\ar[r]&a\wedge b.}
$$
We also define the \emph{based hom object} for $a$ and $b$ to be
the pullback in $\CC$ given in the following diagram:
$$\xymatrix{
\hom_*(a,b)\ar[r]\ar[d]&t\ar[d]
\\ \hom(a,b)\ar[r]&\hom(t,b).}
$$
The arrows in the pullback system are induced by the structure maps
for $a$ and $b$ and the isomorphism $t\cong\hom(t,t)$ that comes from
the fact that $\hom(t,-)$ preserves products (and $t$ is the empty
product).  The composite
$$t\cong\hom(a,t)\to\hom(a,b)\to\hom(t,b)$$
coincides with the given arrow from $t$ to $\hom(t,b)$, so induces
a structure map for $\hom_*(a,b)$ as an object of $\CC_*$.
\end{construction}

\begin{lemma}\label{SmashMonoidal} Construction \ref{smash} makes
$\CC_*$ into a symmetric monoidal closed bicomplete category.
\end{lemma}

\begin{proof} First, $\CC_*$ is bicomplete, being a slice category
of $\CC$.  The definition makes it clear that $\wedge$ is
symmetric.  The rest of the claim, as in the proof of Theorem
\ref{tensormonoidal}, relies on an enriched adjunction,
$$\hom_*(a,\hom_*(b,c))\cong\hom_*(a\wedge b,c),
$$
which we establish first.  Since $\hom(a,x)$ for a constant object
$a$ is a right adjoint, it preserves limits in $x$, and in
particular pullbacks.  Consequently, we can display the left side
of the adjunction we seek as part of the following diagram, in
which there are three pullbacks: the top rectangle, and the left
and right sides of the cubical diagram to which it is connected:
$$\xymatrix@C-40pt{
\hom_*(a,\hom_*(b,c))\ar[rr]\ar[d]
&&t\ar[d]
\\ \hom(a,\hom_*(b,c))\ar[rr]\ar[dr]\ar[dd]
&&\hom(t,\hom_*(b,c))\ar[dr]\ar[dd]
\\&t\ar[rr]\ar[dd]&&t\ar[dd]
\\ \hom(a,\hom(b,c))\ar[rr]\ar[dr]
&&\hom(t,\hom(b,c))\ar[dr]
\\&\hom(a,\hom(t,c))\ar[rr]&&\hom(t,\hom(t,c)).}
$$
Next, observe that $\hom(a,\hom(b,c))\cong\hom(a\otimes b,c)$ in
$\CC$ as a consequence of the associativity of $\otimes$, so we
can rewrite the bottom of our diagram to get
$$\xymatrix@C-40pt{
\hom_*(a,\hom_*(b,c))\ar[rr]\ar[d]
&&t\ar[d]
\\ \hom(a,\hom_*(b,c))\ar[rr]\ar[dr]\ar[dd]
&&\hom(t,\hom_*(b,c))\ar[dr]\ar[dd]
\\&t\ar[rr]\ar[dd]&&t\ar[dd]
\\ \hom(a\otimes b,c)\ar[rr]\ar[dr]
&&\hom(t\otimes b,c)\ar[dr]
\\&\hom(a\otimes t,c)\ar[rr]&&\hom(t\otimes t,c).}
$$
Now observe that on the right side of the diagram, we have the
vertical composite
$$t\to\hom(t,\hom_*(b,c))\to\hom(t\otimes b,c),
$$
which coincides with
$$t\cong\hom(t\otimes b,t)\to\hom(t\otimes b,c).
$$
Consequently, the diagram actually displays
$\hom_*(a,\hom_*(b,c))$ as the limit of the following diagram:
$$\xymatrix{
\hom(a\otimes b,c)\ar[drr]\ar[ddr]
\\&t\ar[r]\ar[d]&\hom(t\otimes b,c)\ar[d]
\\&\hom(a\otimes t,c)\ar[r]&\hom(t\otimes t,c).}
$$
Since both squares commute, we can remove the $\hom(t\otimes t,c)$
and consequently have $\hom_*(a,\hom_*(b,c))$ as the limit of the
smaller diagram
$$\xymatrix{
\hom(a\otimes b,c)\ar[r]\ar[d]&\hom(t\otimes b,c)
\\ \hom(a\otimes t,c)&t.\ar[l]\ar[u]}
$$
Since $\otimes$ preserves coproducts, being a left adjoint, and is
symmetric, we have
$$\hom(a\otimes t,c)\times\hom(t\otimes b,c)\cong\hom((a\otimes
t)\amalg(t\otimes b),c),
$$
so the pairs of arrows out of a single source can be combined, and
we can display $\hom_*(a,\hom_*(b,c))$ as the pullback in the
diagram
$$\xymatrix{
\hom_*(a,\hom_*(b,c))\ar[r]\ar[d]
&\hom(a\otimes b,c)\ar[d]
\\ t\ar[r]&\hom((a\otimes t)\amalg(t\otimes b),c).}
$$
Next, $\hom(x,c)$ sends colimits in $x$ to limits, again by
adjointness, so the pushout defining $a\wedge b$ gives us a
pullback
$$\xymatrix{
\hom(a\wedge b,c)\ar[r]\ar[d]&\hom(a\otimes b,c)\ar[d]
\\ \hom(t,c)\ar[r]&\hom((a\otimes t)\amalg(t\otimes b),c).}
$$
This in turn pastes onto the pullback diagram defining $\hom_*$,
giving us a composite pullback
$$\xymatrix{
\hom_*(a\wedge b,c)\ar[r]\ar[d]
&\hom(a\wedge b,c)\ar[r]\ar[d] &\hom(a\otimes b,c)\ar[d]
\\t\ar[r]&\hom(t,c)\ar[r]
&\hom((a\otimes t)\amalg(t\otimes b),c).}
$$
By the uniqueness of pullbacks, we get the enriched adjunction we
claimed.

Next, the same argument, but with the outer $\hom$'s replaced with
$\CC$'s and the outer $\hom_*$'s replaced with $\CC_*$'s shows
that $\hom_*$ really is right adjoint to $\wedge$, i.e.,
$$\CC_*(a\wedge b,c)\cong\CC_*(a,\hom_*(b,c))
$$
natural in $a$, $b$, and $c$; note that by definition,
$$\xymatrix{
\CC_*(a,b)\ar[r]\ar[d]&{*}\ar[d]
\\ \CC(a,b)\ar[r]&\CC(t,b)}
$$
is a pullback (of sets.)  For associativity of $\wedge$, we can
now use the Yoneda Lemma:
\begin{gather}
\CC_*((a\wedge b)\wedge c,d)\cong\CC_*(a\wedge b,\hom_*(c,d))
\cong\CC_*(a,\hom_*(b,\hom_*(c,d)))\notag
\\ \cong\CC_*(a,\hom_*(b\wedge c,d)) \cong\CC_*(a\wedge(b\wedge
c),d).\notag
\end{gather}
Consequently, $(a\wedge b)\wedge c\cong a\wedge(b\wedge c)$,
naturally in $a$, $b$, and $c$.  The unit for $\wedge$ is easily
seen to be $e\amalg t$, where $e$ is the unit for $\otimes$. This
concludes the proof.
\end{proof}

\begin{corollary}\label{SMC} The category $\Multstar$ of pointed
multicategories is bicomplete and symmetric monoidal closed using
this smash product construction.
\end{corollary}

We leave to the reader the straightforward task of verifying that
the underlying multicategory structure for this symmetric monoidal
structure on $\Multstar$ coincides with the one specified in
Section \ref{mult}.

\section{The $K$-theory of multicategories}
\label{Kmult}

This section is devoted to the description of our lax symmetric monoidal
$K$-theory functor from $\Multstar$ to symmetric spectra, and the
following section will show it is consistent with the $K$-theory
of permutative categories described in \cite{RMA}.

As mentioned in the introduction, our construction is the
composite of two functors, with the intermediate category being
the category of $\Gstar$-categories introduced in \cite{RMA},
Section 5 (and given a simplified description here), and with the
functor from $\GstarCat$ to symmetric spectra being the one
described in
\cite{RMA}, Section 7.  We are therefore left with the task of
describing a lax symmetric monoidal functor $\Jhat$ from
$\Multstar$ to $\GstarCat$, and from the construction, this
functor will actually be representable in a sense that we will
make clear below.

We begin with the following general categorical proposition, which
we will need in two separate places. We would like to thank Mike
Shulman for pointing it out to us, as well as for noting that it
is a special case of
\cite{Bor}, Proposition 4.2.3.  We provide a brief proof (due to
Shulman) for the convenience of the reader.

\begin{proposition}\label{coreprops}
Let $(\C,\wedge,1)$ be a monoidal category, $A$ an object of $\C$,
and let $\mu:A\wedge A\to A$ be an isomorphism providing the
multiplicative structure map for $A$ as a monoid in $\C$.  Then
for any left $A$-module $M$, the structure map $\xi:A\wedge M\to
M$ is an isomorphism, the category of left $A$-modules is a full
subcategory of $\C$, and similarly for right $A$-modules. Further,
given a left $A$-module $M$ and a right $A$-module $N$, the two
maps
$$\xymatrix{
N\wedge A\wedge M\ar@<.5ex>[r]^-{\xi_N\wedge1}
\ar@<-.5ex>[r]_-{1\wedge\xi_M}
&N\wedge M}
$$
coincide, so the canonical map $N\wedge M\to N\wedge_A M$ is an
isomorphism.
\end{proposition}

\begin{proof}
Let $\eta:1\to A$ be the unit map for $A$ as a monoid in $\C$.
Then the commutativity of
$$\xymatrix{
A\ar[r]^-{\eta\wedge1}\ar[dr]_-{=} &A\wedge
A\ar[d]^-{\mu}_-{\cong}&A\ar[l]_-{1\wedge\eta}\ar[dl]^-{=}
\\&A}
$$
shows that $\eta\wedge1=\mu^{-1}=1\wedge\eta$.  Now let $M$ be a
left $A$-module with structure map $\xi:A\wedge M\to M$.  Then the
commutative naturality diagram
$$\xymatrix{
A\wedge M\ar[r]^-{\eta\wedge1\wedge1}\ar[d]_-{\xi} &A\wedge
A\wedge M\ar[d]^-{1\wedge\xi}
\\M\ar[r]_-{\eta\wedge1}&A\wedge M}
$$
can be rewritten by replacing $\eta\wedge1$ with $1\wedge\eta$ in
the top arrow, resulting in
$$\xymatrix{
A\wedge M\ar[r]^-{1\wedge\eta\wedge1}\ar[d]_-{\xi} &A\wedge
A\wedge M\ar[d]^-{1\wedge\xi}
\\M\ar[r]_-{\eta\wedge1}&A\wedge M.}
$$
Now since $\xi\circ(\eta\wedge1)=\id_M$, the clockwise composite
is $\id_{A\wedge M}$, which shows that
$(\eta\wedge1)\circ\xi=\id_{A\wedge M}$, so
$\eta\wedge1=\xi^{-1}$, and $\xi$ is an isomorphism.  The
analogous argument holds for right $A$-modules.

Next, the naturality of $\eta$ shows that for all $f:M\to N$ in
$\C$, the following diagram commutes:
$$\xymatrix{
A\wedge M\ar[r]^-{1\wedge f}&A\wedge N
\\M\ar[u]^-{\eta\wedge1}_{\cong}\ar[r]^-{f}
&N.\ar[u]^-{\cong}_-{\eta\wedge1}}
$$
But since $\eta\wedge1=\xi^{-1}$ for left $A$-modules, it follows
that if $M$ and $N$ are left $A$-modules,
$$\xymatrix{
A\wedge M\ar[r]^-{1\wedge f}\ar[d]_-{\xi}&A\wedge N\ar[d]^-{\xi}
\\M\ar[r]^-{f}&N}
$$
commutes.  Consequently left $A$-modules form a full subcategory
of $\C$, and similarly for right $A$-modules.

Finally, given a left $A$-module $M$ and a right $A$-module $N$,
$$\xymatrix{
N\wedge M\ar[r]^-{1\wedge\eta\wedge1}&N\wedge A\wedge M}
$$
is inverse to both $1\wedge\xi_M$ and $\xi_N\wedge1$, so the two
maps coincide.
\end{proof}

Our first use of Proposition \ref{coreprops} will be to formalize some machinery involving the smash product of based categories; this in turn is used to describe the category $\Gstar$ and the category of $\Gstar$-categories.

From Lemma \ref{SmashMonoidal}, we know the category $\Catstar$ of
based categories is symmetric monoidal, closed, and bicomplete,
where a based category is simply a category with a selected
object.  In particular, there is \emph{no} property the basepoint
object must satisfy.  On the other hand, when we require the basepoint
object to be null (initial and final), the morphism sets become based,
and we have the following straightforward description of the smash
product.

\begin{proposition}\label{nullsmash}
If $\C$ and $\D$ are based categories with
null basepoint objects, then:
\begin{enumerate}
\item $\C\wedge\D$ has null basepoint
object.
\item $\Ob(\C\wedge\D)\cong \Ob\C\wedge \Ob\D$
\item For
any objects $a_1$, $a_2$ of $\C$, and objects $b_1$, $b_2$ of
$\D$,
$$(\C\wedge\D)((a_1,b_1),(a_2,b_2))
\cong\C(a_1,a_2)\wedge\D(b_1,b_2).
$$
\end{enumerate}
\end{proposition}

\begin{proof}
We can construct a category $\B$ by $\Ob\B=\Ob\C\wedge\Ob\D$ and
$\B((a_{1},b_{1}),(a_{2},b_{2}))=\C(a_1,a_2)\wedge\D(b_1,b_2)$,
with composition and identities defined by composition and
identities in $\C$ and $\D$.  We then have a canonical functor
$\C\times \D\to \B$ and to see that it satisfies the universal
property defining the smash product the only issue is whether,
given morphisms $\phi\in\C(a_1,a_2)$ and $\psi\in\D(b_1,b_2)$, we
have
$$g(\phi,*)=g(*,*)=g(*,\psi).
$$
Consider the diagram
$$\xymatrix{
g(a_1,b_1)\ar[d]^-{g(\phi,1)}\ar[r]
&g(a_1,*)\ar[d]^-{=}\ar[r]
&g(a_1,b_2)\ar[d]^-{g(\phi,1)}
\\g(a_2,b_1)\ar[r]
&g(a_2,*)\ar[r]
&g(a_2,b_2).}
$$
This shows that $g(\phi,*)$ coincides with the composite
$$g(a_1,b_1)\to g(a_1,*)=g(a_2,*)\to g(a_2,b_2),
$$
which is independent of $\phi$.  Therefore $g(\phi,*)=g(*,*)$. A
similar diagram shows that $g(*,\psi)=g(*,*)$.
\end{proof}

The unit for the smash product in $\Catstar$ is the two object
discrete category $S^{0}$; it has two objects, namely a basepoint $*$
and a non-basepoint $()$, and has only identity morphisms. Note that
the basepoint object of $S^{0}$ is not null.  On the other hand, we
can construct a unit in the full subcategory of based categories with
null basepoint object as follows.

\begin{definition}
Let $e$ be the based category with two
objects, $*$ and $()$, with $*$ a null basepoint object, and with the
set of self maps of $()$ consisting of the null map and the
identity.
\end{definition}

The following theorem is essentially a corollary of Propositions
\ref{coreprops} and \ref{nullsmash}.

\begin{theorem}\label{emod}
The based category $e$ satisfies $e\wedge e\cong
e$, with the isomorphism making $e$ a commutative monoid in
$\Catstar$.  The category of $e$-modules is precisely the full
subcategory of $\Catstar$ of based categories with a null basepoint
object, and the smash product over $e$ is naturally isomorphic to the
smash product in $\Catstar$.
\end{theorem}

Although the smash products are the same, the units are
different: $e$ is the unit for the category of $e$-modules,
and as mentioned above, $S^{0}$ is the unit for $\Catstar$.

\begin{proof}
The only claim that does not follow immediately from Propositions \ref{coreprops} and
\ref{nullsmash} is the identification of $e$-modules with based
categories having null basepoints.  If $\C$ has a null basepoint,
then it is easy to produce a unique $e$-module structure map.
Conversely, suppose $\C$ is an $e$-module.  Then $\C$ supports a
based functor $\xi:e\times\C\to\C$, which we claim is split
epi. This follows from the fact that the induced map on smash
products $e\wedge\C\to\C$ is unital, using the following diagram:
$$\xymatrix{
S^0\times\C\ar@<0ex>[r]\ar[d]&S^0\wedge\C\ar@{.>}@<-.8ex>[l]\ar[dr]^-{\cong}\ar[d]
\\ e\times\C\ar[r]&e\wedge\C\ar[r]&\C.}
$$
The vertical arrows are induced by the inclusion $S^0\to e$, and
the top rightward arrow is split by observing
$S^0\wedge\C\cong\C\cong\{()\}\times\C$ and including $\{()\}$
into $S^0$. (Here, of course, $\{()\}$ is a one point category
with object $()$.)  Now it follows that the bottom composite,
which is $\xi$, splits, and that the splitting is the composite
$$\C\cong\{()\}\times\C\to e\times\C.
$$

Now suppose $*$ is the basepoint object in $\C$, and let $\phi:*\to a$
be any map from the basepoint.  Then we can consider the morphism
$\xi(*\to(),\phi):\xi(*,*)\to\xi((),a)$.  Since $\xi$ is a
bifunctor, we have the commutative square with both composites
being $\xi(*\to(),\phi)$:
$$\xymatrix@C+20pt{
\xi(*,*)\ar[r]^-{\xi(*\to(),*)}\ar[d]^-{\xi(*,\phi)}
&\xi((),*)\ar[d]^-{\xi((),\phi)}
\\ \xi(*,a)\ar[r]^-{\xi(*\to(),a)}
&\xi((),a).}
$$
But since $\xi$ is based, both of the arrows out of the
top left entry are the identity on $*$, and $\xi((),\phi)=\phi$
and $\xi((),a)=a$ follow from the splitting.  The diagram then
tells us that
$$\phi=\xi((),\phi)=\xi(*\to(),a),
$$
which is independent of $\phi$. Therefore the only morphism from
$*$ to $a$ is $\xi(*\to(),a)$, so $*$ is initial.  Similarly, $*$
is terminal, and therefore null.
\end{proof}

We are now ready to describe our indexing based category $\Gstar$;
it is constructed as a Grothendieck construction in the category
of $e$-modules.  Since this may be somewhat exotic for some
readers, we provide the details.

Let $\F$ be the
skeleton of the category of finite based sets consisting of the
objects ${\b n}=\{0,1,\dotsc,n\}$ with basepoint 0. We observe that $\F$ is a category with a null basepoint,
and therefore an $e$-module.  We write $\F^{(r)}$ for the $r$-th
smash power of $\F$ as an $e$-module, so in particular
$\F^{(0)}=e$.  (All other smash powers are formed in $\Catstar$,
by Theorem~\ref{emod}.)  We write objects of
$\F^{(r)}$ as $\br{\b m}=(\b m_1,\dots,\b m_r)$, with the understanding that $\br{\b
m}=*$, the basepoint, if any ${\b m}_i=0$.

Now let $\Inj$ be the category
with objects the unbased sets $\u r=\{1,\dotsc,r\}$ for
$r=0,1,2,3,\ldots$, and morphisms the injections. The categories
$\F^{(r)}$ are the target objects of a functor
\[
\F^{(-)}:\Inj\to e\text{-\bf{mod}}
\]
taking $\u r$ to $\F^{(r)}$ on objects, so in particular $\u 0$ gets sent to $\F^{(0)}=e$. On morphisms, $\F^{(-)}$
rearranges the coordinates according to the given injection and,
most crucially, inserts the object $\b 1$ in the slots that are
missed: the intuition is that the objects of $\F^{(r)}$ are lists of objects of $\F$ that are waiting to be smashed together, and the injections merely rearrange the lists without affecting the size of the smash product.  Formally, if we are given an injection $q:\u r\to\u s$,
then $\F^{(q)}$ is the functor from $\F^{(r)}$ to $\F^{(s)}$ that takes a non-basepoint
object $\br{\b m}=(\b m_1,\dotsc,\b m_r)$ to the $s$-tuple
$q_*\br{\b m}=(\b m'_1,\dotsc,\b m'_s)$ in which
$$\b m'_j=
\begin{cases}\b m_i&\text{if $q^{-1}(j)=\{i\}$}\\
\b 1&\text{if $q^{-1}(j)=\emptyset$,}
\end{cases}
$$
and takes a morphism $(\alpha_1,\dotsc,\alpha_r)$ to the $s$-tuple
$(\alpha'_1,\dotsc,\alpha'_s)$ where
$$\alpha'_j=
\begin{cases} \alpha_i&\text{if $q^{-1}(j)=\{i\}$}\\
\id_{\b 1}&\text{if $q^{-1}(j)=\emptyset$.}
\end{cases}
$$
 in particular,
the object $()$ of $\F^{(0)}$ gets sent to the constant string
$({\b1},\dotsc,{\b1})$. We define $\Gstar$ to be the Grothendieck construction $\Inj\int \F^{(-)}$ formed in $e$-modules: this is formally the same as the ordinary Grothendieck construction, except we use the coproduct in $e$-modules (which is a wedge) instead of the coproduct in $\Cat$ (which is a disjoint union), and we use the smash product of based sets and categories instead of the Cartesian product.  Specifically,
the set of objects of $\Gstar$ is
$$ \bigvee_{\u r\in \obj{\Inj}} \obj{\F^{(r)}} $$
and the maps from $\br{\b m}$ to $\br{\b n}$ in $\Gstar$ form the
based set
\[
\bigvee_{q\colon \u r \to \u s}
\left(\bigwedge_{j=1}^s \F(\b m_{q^{-1}(j)},\b n_j)\right),
\]
where we agree that if $q^{-1}(j)=\emptyset$, then
$\b m_{q^{-1}(j)}=\b 1$.
The empty wedge
is of course the one point set, and the empty smash is $\b 1$.
Note that the basepoint object $*$ of $\Gstar$ is a null object,
and the basepoint in each mapping set is the unique map that
factors through $*$.

For readers of \cite{RMA}, we note that this is a based version of
the category $\G$ introduced there, and that we have the following
relation between $\G$ and $\Gstar$.  First, there is a canonical
functor $\G\to \Gstar$.  More specifically, we can identify the
category $\Gstar$ as the category obtained from $\G$ by attaching
a new null object $*$ and identifying $\br{\b m}$ with $*$
whenever any ${\b m}_{i}=\b0$.  In particular, every map in
$\Gstar(\br{\b m},\br{\b n})$ is either the trivial morphism
(factoring through $*$) or in the image of $\G(\br{\b m},\br{\b
n})$.  Whenever none of the entries in $\br{\b m}$ or $\br{\b n}$
are $\b0$, the function $\G(\br{\b m},\br{\b n})\to
\Gstar(\br{\b m},\br{\b n})$ is in fact one-to-one onto the subset
of $\Gstar(\br{\b m},\br{\b n})$ that excludes the trivial
morphism.

In order to avoid possible
confusion as to the meaning of ``functor'' and ``natural
transformation'' where they occur below, we define
$\Gstar$-objects in any category $\ACat$ with a final object, which we
always denote as $*$.
Let $\ACat_*$ be the category of objects under $*$, so $\ACat_*$ has
$*$ as a null object. A
\emph{based functor}\/ $\Gstar\to \ACat_{*}$ is a functor that
takes the null object $*$ of $\Gstar$ to the null object $*$ of
$\ACat_{*}$.  Our intermediate category $\GstarCat$ is a special case of the following definition:

\begin{definition}
The category $\GC$ is the category of based functors
$\Gstar\to \ACat_{*}$.
\end{definition}

Concatenation of lists makes $\Gstar$ into a permutative category,
where concatenation with $*$ on either side yields $*$; we denote
this operation by $\concat$. It follows from theorems of Day
(\cite{Day}, Theorems 3.3 and 3.6) that when $\ACat_{*}$ is a
bicomplete closed
symmetric monoidal category, the category of based
functors from $\Gstar$ to $\ACat_*$ has a closed symmetric
monoidal structure, enriched over $\ACat_*$, in which the product
of functors $F_{1}$ and $F_{2}$ is given by the left Kan extension
$F_{1}\sma F_{2}$ in the diagram on the left below.   The
universal property of the Kan extension is that maps from
$F_{1}\sma F_{2}$ to $G$ are in one-to-one correspondence with
natural transformations $f$ as  in the diagram on the right below:
$$\xymatrix{\Gstar\times\Gstar\ar[r]^-{F_{1}\times F_{2}}\ar[d]_-{\concat}
&\ACat_*\times\ACat_*\ar[r]^-{\sma}
&\ACat_*
\\ \Gstar\ar@{-->}[urr]_-{F_{1}\sma F_{2}}}
\qquad
\xymatrix@C=40pt{
\Gstar\times\Gstar \ar[r]^-{F_1\times F_2}\ar[d]_-{\concat}
\drtwocell<\omit>{f}
&\ACat_*\times\ACat_{*}\ar[d]^-{\sma}
\\ \Gstar\ar[r]_-{G}
&\ACat_*.}
$$
This then gives us the following theorem.

\begin{theorem}
Let $\ACat$ be a bicomplete closed symmetric monoidal category.  Then $\GC$
is a closed symmetric monoidal category enriched over $\ACat_{*}$.
\end{theorem}

We are now ready to build our representable lax symmetric monoidal
functor $\Jhat$ from $\Multstar$ to $\GstarCat$, and we do so by
producing its representing object $\Estar$.  The correct formal
structure $\Estar$ must exhibit is firstly that of a based
$\Gstar^\op$-multicategory, that is, a based contravariant functor
from $\Gstar$ to $\Multstar$.  Since $\Multstar$ is symmetric
monoidal closed, and therefore enriched over itself, the lax
symmetric monoidal forgetful functor $\Multstar\to\Catstar$ gives
$\Multstar$ an enrichment over $\Catstar$, and therefore
$$\Multstar(\Estar,M)
$$
the structure of a $\Gstar$-category for any based multicategory
$M$.  Secondly, $\Estar$ needs additional structure to ensure that
the functor it represents is lax symmetric monoidal; we will address this
issue as well.

Our construction of $\Estar$ is based on that of a very small
based multicategory $E$ with excellent formal properties
reminiscent of the based category $e$.

\begin{definition} The multicategory $E$ has two objects, 0 and 1,
with morphisms given by
$$E(a_1,\dots,a_k;a')=
\begin{cases}*&\text{if $a_1+\dots+a_k=a'$}
\\ \emptyset&\text{otherwise,}
\end{cases}
$$
so in particular there are no morphisms when there is more than
one input with value 1. The object 0 is the basepoint object, given by
the unique multifunctor $*\to E$.
\end{definition}

We remark that $E$ is the terminal parameter multicategory for
modules (\cite{RMA}, definition 2.4), so a multifunctor from $E$
gives the image of 0 the structure of a commutative monoid, and
the image of 1 the structure of a module over this monoid.  In the
case where the target is a based multicategory, we already have a
selected commutative monoid structure on the basepoint object, and a
based multifunctor from $E$ is then the choice of a module
structure over this commutative monoid.

The formal properties we need for $E$ are the following.

\begin{theorem}\label{Emods}
The based multicategory $E$ satisfies $E\wedge E\cong E$, with the
isomorphism making $E$ a commutative monoid in $\Multstar$.  The
category of $E$-modules is therefore a full subcategory of
$\Multstar$, and has the same smash product as in $\Multstar$.
\end{theorem}

The Theorem is a special case of Proposition \ref{coreprops}, so
we need only show that $E$ satisfies the hypotheses of Proposition
\ref{coreprops}. We begin with the following lemma.

\begin{lemma}\label{balanced}
Let $M$ be a based multicategory, $f:(E,E)\to M$ a based bilinear
map.  Then for any object $a$ of $E$ and morphism $\phi$ of $E$,
$$f(a,\phi)=f(\phi,a).
$$
\end{lemma}

\begin{proof}
If  $a=0$ or $\phi$ is in the image of $*\to E$, the lemma follows
from $f$ being based.  We consider next the first case in which
neither is true.  Let $\phi_2\in E_2(0,1;1)$ be the unique
element; we wish to show that
$$f(1,\phi_2)=f(\phi_2,1).
$$
The key to the argument is to observe that bilinearity means in
particular that the following diagram commutes:
$$\xymatrix@C+40pt{
(f(0,0),f(0,1),f(1,0),f(1,1))\ar[r]^-{(f(0,\phi_2),f(1,\phi_2))}
\ar[d]^-{\cong}
&(f(0,1),f(1,1)) \ar[dd]^-{f(\phi_2,1)}
\\(f(0,0),f(1,0),f(0,1),f(1,1))\ar[d]^-{(f(\phi_2,0),f(\phi_2,1))}
\\(f(1,0),f(1,1))\ar[r]^-{f(1,\phi_2)}
&f(1,1).}
$$

We now precompose with the ordered quadruple of morphisms
$(\id_{f(0,0)},\e_0,\e_0,\id_{f(1,1)})$, where $\e_0$ is the image
in $M$ of the canonical $0$-morphism in $*$. Notice that
$f(0,0)=f(0,1)=f(1,0)$ since all must be the basepoint of $M$, and
that the composite
$$\xymatrix@C+20pt{
f(0,0)\ar[r]^-{(\id_{f(0,0)},\e_0)}
&(f(0,0),f(1,0))\ar[r]^-{f(\phi_2,0)} &f(1,0),}
$$
being in the image of $*\to M$, must be the identity $1$-morphism.
Similarly, the composite
$$\xymatrix@C+20pt{
f(0,0)\ar[r]^-{(\id_{f(0,0)},\e_0)} &(f(0,0),f(0,1))
\ar[r]^-{f(0,\phi_2)} &f(0,1)}
$$
is the identity.  Furthermore, the composite
$$\xymatrix@C+20pt{
f(1,1)\ar[r]^-{(\e_0,\id_{f(1,1)})}
 &(f(1,0),f(1,1))\ar[r]^-{f(1,\phi_2)}
 &f(1,1)}
$$
arises from applying the based multifunctor $f(1,\blank)$ to the
composite
$$\xymatrix{
1\ar[r]^-{(\e_0,\id)}
 &(0,1)\ar[r]^-{\phi_2}
 &1,}
$$
in $E$, which is $\id_1$, and therefore the previous composite is
$\id_{f(1,1)}$.  Similarly, the composite
$$\xymatrix@C+20pt{
f(1,1)\ar[r]^-{(\e_0,\id)}
 &(f(0,1),f(1,1))\ar[r]^-{f(\phi_2,1)}
 &f(1,1)}
$$
is also $\id_{f(1,1)}$.  We now have the total diagram
$$\xymatrix{
(f(0,0),f(1,1))\ar[r]^-{=}\ar[d]^-{=}
 &(f(0,1),f(1,1))\ar[d]^-{f(\phi_2,1)}
 \\(f(1,0),f(1,1))\ar[r]^-{f(1,\phi_2)}
 &f(1,1)}
 $$
 which establishes the claim.

 We next consider the unique element $\phi_n\in E_n(0^{n-1},1;1)$,
 and claim that
 $$f(1,\phi_n)=f(\phi_n,1).
 $$
 This follows by induction from the case $n=2$ by use of the
 multifunctoriality of $f(1,\blank)$ and $f(\blank,1)$, together
 with the formula
 $$\phi_n=\Gamma(\phi_2;\e_{n-1},\phi_2),
 $$
 where $\e_{n-1}$ is the canonical $(n-1)$-morphism $0^{n-1}\to 0$.
 The general case now follows, since all morphisms in $E$ are
 either part of the basepoint structure or else arise from a
 permutation action on one of the $\phi_n$'s.
\end{proof}

\begin{corollary}\label{Ecore}
There is a natural isomorphism $E\wedge E\cong E$ which is the
product map for a commutative monoid structure on $E$ in
$\Multstar$.
\end{corollary}

\begin{proof}
The isomorphism is induced from the obvious based bilinear map
$$\beta:(E,E)\to E
$$
sending $(1,1)$ to 1 and both $(1,\phi)$ and $(\phi,1)$ to $\phi$
for any morphism $\phi$.  If $f:(E,E)\to M$ is any other based
bilinear map, then Lemma \ref{balanced} shows that $f$ factors
uniquely through $\beta$, giving $\beta$ the universal property of
the map to the smash product.  The isomorphism now follows from
the uniqueness of universal objects.  The axioms for the
commutative monoid structure are trivial to verify, and follow
from the fact that $E$ has no nontrivial automorphisms.
\end{proof}

The proof of Theorem \ref{Emods} now consists of Proposition
\ref{coreprops} applied to Corollary \ref{Ecore}.

Note in particular that the unit for the smash product of based
multicategories, which we will call $u$, is the coproduct
$U\coprod*$, where $*$ is the terminal multicategory and $U$ is
the unit for the tensor product in $\Mult$, which has one object
and only its identity morphism. We will think of $u$ as having two
objects, 0 and 1, with 0 the basepoint object, and with the only
morphism involving 1 being $\id_1$.  It is now clear what the unit
map $u\to E$ is.

The formal properties of our representing object $\Estar$, which
is still to be defined, rely on those of the Cartesian power
multicategories $E^n$, for which we first need some notation. The
multicategories $E^n$ are the powers using the Cartesian product
of multicategories, which provides the categorical product in both
the based and the unbased settings. It is formed using the
Cartesian product of sets on both objects and $k$-morphisms for
each $k$.  We will find it convenient to think of an object of
$E^n$, which is merely a string of 0's and 1's of length $n$, as
being given by the subset $T\subset\{1,\dots,n\}$ of indices at
which the string takes on the value 1.  With this in mind, it is
easy to verify the following proposition.

\begin{proposition}  Given objects $T_1,\dots,T_k$ and $T'$ of
$E^n$, the set of $k$-morphisms $E^n(T_1,\dots,T_k;T')$ is empty
unless the $T_i$'s are mutually disjoint and $T_1\cup\dots\cup
T_k=T'$, in which case it consists of a single $k$-morphism.
\end{proposition}

Our next step in deriving the formal properties of $E^*$ is the
following structure theorem about cartesian powers of $E$.

\begin{theorem}\label{Emod}
The cartesian powers $E^m$ are modules over the commutative monoid
$E$ in $\Multstar$.
\end{theorem}

\begin{proof} We define the module structure map by giving its
associated bilinear map; on objects we do the only possible thing:
given an object $S$ of $E^m$, we send $(1,S)$ to $S$.  On
morphisms, we send $(1,\phi)$ to $\phi$, and given the
$k$-morphism $\phi_k:(0^{k-1},1)\to1$ in $E$, we send $(\phi_k,S)$
to the single $k$-morphism $\phi_k^S:(\emptyset^{k-1},S)\to S$ in
$E^m$. All other assignments are now forced by equivariance.  It
is easy to verify the requirements for a module structure.
\end{proof}

We are now ready to define our $\Gstar^\op$-multicategory
$\Estar$.

\begin{definition} Given an object $\br{\b m}=({\b m_1},\dots,
{\b m_k})$ of $\Gstar$, we define $E^*\br{\b m}$ to be
$E^{m_1}\wedge\dots\wedge E^{m_k}$, where $E^m$ is the $m$'th
cartesian power of $E$.  In particular, the $0$-th Cartesian power
$E^0$ is $*$, the null multicategory in $\Multstar$, which also
acts as a 0 object for the smash product in $\Multstar$.  We
define $\Estar()=E$.
\end{definition}

\begin{theorem}
$\Estar$ supports the structure of a $\Gstar^\op$-multicategory.
\end{theorem}

\begin{proof}
Suppose given a non-basepoint morphism $(\alpha,q):(\b
m_1,\dots,\b m_r)\to(\b n_1,\dots,\b n_s)$ in $\Gstar$.  We must
define $E^*(\alpha, q):E^*\br{\b n}\to E^*\br{\b m}$ compatible
with the composition in $\Gstar$.

For each $j$ with $1\le j\le s$, we have a given morphism in $\F$
$$\alpha_j:\b m_{q^{-1}(j)}\to\b n_j,
$$
where we have $\b m_{q^{-1}(j)}=\b1$ if $q^{-1}(j)=\emptyset$.
These induce maps of based multicategories
$$\alpha_j^*:E^{n_j}\to E^{m_{q^{-1}(j)}}
$$
by requiring the maps to fit into commutative diagrams with the
product projection maps for $1\le t\le m_{q^{-1}(j)}$:
$$\xymatrix{
E^{n_j}\ar[r]^-{\alpha_j^*}\ar[dr]_-{\pi_{\alpha_j(t)}}
&E^{m_{q^{-1}(j)}}\ar[d]^-{\pi_t}
\\&E,}
$$
where $\pi_{\alpha_j(t)}=*$, the null map, if $\alpha_j(t)=0$.

Now smashing the $\alpha_j^*$'s together gives us a map
$$\alpha^*:E\br{\b n}=E^{n_1}\wedge\dots\wedge E^{n_s}\to
E^{m_{q^{-1}(1)}}\wedge\dots\wedge E^{m_{q^{-1}(s)}} =E(q_*\br{\b
m}).
$$
Further, Theorem \ref{Emod} gives $E$-module structure maps for
the Cartesian powers $E^{m_i}$, while Proposition
\ref{coreprops} and Theorem \ref{Emods} show that the order in which the factors
$E=E^{m_{q^{-1}(j)}}$ for $q^{-1}(j)=\emptyset$ are absorbed is
immaterial. Consequently, we get a canonical isomorphism
$$\xymatrix{
q^*:E^*(q_*\br{\b m})\ar[r]^-{\cong}&E^*\br{\b m},}
$$
and we define $E^*(\alpha,q)=q^*\circ\alpha^*$.  The verification
that this definition is compatible with composition in $\G$ is
left to the reader.
\end{proof}

This Theorem now justifies the following definition.

\begin{definition}
Define $\Jhat\colon \Multstar \to \GstarCat$ by letting
$\Jhat{M}\br{\b m}$ be the underlying based category of
$\Multstar(\Estar\br{\b m},M)$.
\end{definition}

Finally, we must show that $\Jhat$ is lax symmetric monoidal.  The
existence of the lax structure map for the product follows from
two observations: first, given objects $\br{\b m}$ and  $\br{\b
n}$ of $\Gstar$, we have
$$\Estar\br{\b m}\wedge\Estar\br{\b n}=\Estar(\br{\b
m}\concat\br{\b n}),
$$
and second, the definition of the smash product of
$\Gstar$-categories as a Kan extension makes it only necessary to
observe that we have a natural map
\begin{gather}
\Mult_*(E^*\br{\b m},M)\times\Mult_*(E^*\br{\b n},N)\notag
\\ \to\Mult_*(E^*\br{\b m}\wedge  E^*\br{\b n},M\wedge N)\notag
\\=\Mult_*(E^*(\br{\b m}\concat\br{\b n}),M\wedge N).\notag
\end{gather}
The structure map for the unit follows from the observation that
$\Multstar(E,u)=*$, the terminal based multicategory, and it
follows that $\Jhat(u)=*$. The lax structure map for the unit is
then given by the unique map to the terminal object. The necessary
coherence properties for a lax symmetric monoidal functor are now
easily verified.

\section{Proof of Consistency}\label{extend}

This section completes the proof of Theorem \ref{extn} by showing
that composing our forgetful multifunctor $U:\P\to\Multstar$ with
the represented lax symmetric monoidal functor
$$\Jhat:\Multstar\to\GstarCat
$$
results in the multifunctor $J$ described in \cite{RMA}, up to
natural isomorphism. Let $\C$ be a permutative category. We begin
by recalling the definition of $J\C$, which assigns to each object
$(\b n_1,\dots,\b n_k)=\br{\b n}$ of $\Gstar$ a category
$J\C\br{\b n}$, which has as its objects systems of objects of
$\C$ indexed by $k$-tuples $\br{S}=(S_1,\dots,S_k)$ of subsets
$S_i\subset\{1,\dots,n_i\}$. Of course, $J$ assigns the null based
category $*$ to the base object of $\Gstar$. (The description of
$J\C$ in
\cite{RMA} is in terms of the category $\G$ which we have not
defined here, but the descriptions are equivalent.) In order to
explain the properties we require for these systems, we need some
notation: given a subset $T\subset\{1,\dots,n_i\}$ for some $1\le
i\le k$, we write $\br{S\subst{i}{T}}$ for the $n$-tuple
$(S_{1},\dotsc,S_{i-1},T,S_{i+1},\dotsc,S_{k})$ obtained by
substituting $T$ in the $i$-th position.  We can now make sense of
the following definition.

\begin{definition}
Let $\C$ be a permutative category and $\br{\b n}=(\b n_1,\dots,\b
n_k)$ a non-basepoint object of $\Gstar$.  The category $J\C\br{\b
n}$ has objects the systems $\{C_{\br{S}},
\rho_{\br{S};i,T,U}\}$, where
\begin{enumerate}
\item $\br{S}=(S_{1},\dotsc,S_{k})$ runs through all $k$-tuples of
subsets $S_{i}\subset\{1,\dots,n_i\}$,
\item For $\rho_{\br{S};i,T,U}$, $i$ runs through $1,\dotsc,k$, and
$T,U$ run through the subsets of $S_{i}$ with $T\cap U=\emptyset$
and $T\cup U = S_{i}$,
\item The $C_{\br{S}}$ are objects of $\C$, and
\item The $\rho_{\br{S};i,T,U}$ are morphisms $C_{\br{S\subst{i}T}}\oplus
C_{\br{S\subst{i}U}}\to C_{\br{S}}$ in $\C$,
\end{enumerate}
such that
\begin{enumerate}
\item $C_{\br{S}}=0$ if $S_{i}=\emptyset$ for any $i$,
\item $\rho_{\br{S};i,T,U}=\id$ if any of the $S_{j}$ (for any $j$),
$T$, or $U$ are empty,
\item For all $\rho_{\br{S};i,T,U}$ the following
diagram commutes:
$$ \xymatrix@C+40pt @R+12pt{
C_{\br{S\subst{i}T}}\oplus C_{\br{S\subst{i}U}}
\ar[d]_{\gamma}\ar[r]^-{\rho_{\br{S};i,T,U}}
&C_{\br{S}}\ar @{=}[d]\\
C_{\br{S\subst{i}U}}\oplus C_{\br{S\subst{i}T}}
\ar[r]_-{\rho_{\br{S};i,U,T}}
&C_{\br{S},}
} $$
\item For all $\br{S}$, $i$, and $T,U,V\subset\{1,\dots,n_i\}$ with $T\cup
U\cup V=S_{i}$ and $T$, $U$, and $V$ mutually disjoint, the
following diagram commutes:
$$ \xymatrix@C+50pt @R+12pt{
C_{\br{S\subst{i}T}}\oplus C_{\br{S\subst{i}U}}\oplus
C_{\br{S\subst{i}V}}
\ar[d]_{\id\oplus \rho_{\br{S\subst{i}(U\cup V)};i,U,V}}
\ar[r]^{\rho_{\br{S\subst{i}(T\cup U)};i,T,U}\oplus \id}
&C_{\br{S\subst{i}(T\cup U)}}\oplus C_{\br{S\subst{i}V}}
\ar[d]^{\rho_{\br{S};i,T\cup U,V}}\\
C_{\br{S\subst{i}T}}\oplus C_{\br{S\subst{i}(U\cup V)}}
\ar[r]_{\rho_{\br{S};i,T,U\cup V}}
&C_{\br{S}},
} $$
\item For all $\rho_{\br{S};i,T,U}$ and $\rho_{\br{S};j,V,W}$ with
$i\ne j$, the following diagram commutes:
$$ \xymatrix@C-35pt @R-3pt{
&C_{\br{S\subst{j}V}}\oplus
C_{\br{S\subst{j}W}}
\ar[ddr]^{\rho_{\br{S};j,V,W}}
\\
C_{\br{S\subst{i}T\subst{j}V}}\oplus
C_{\br{S\subst{i}U\subst{j}V}}\oplus
C_{\br{S\subst{i}T\subst{j}W}}\oplus
C_{\br{S\subst{i}U\subst{j}W}}
\ar[ur]^{(\rho_{\br{S\subst{j}V};i,T,U})\oplus
(\rho_{\br{S\subst{j}W};i,T,U})\qquad\qquad }
\ar[dd]_{\id\oplus \gamma \oplus \id}
\\
&&C_{\br{S}}.\\
C_{\br{S\subst{i}T\subst{j}V}}\oplus
C_{\br{S\subst{i}T\subst{j}W}}\oplus
C_{\br{S\subst{i}U\subst{j}V}}\oplus
C_{\br{S\subst{i}U\subst{j}W}}
\ar[dr]_{(\rho_{\br{S\subst{i}T};j,V,W})\oplus
(\rho_{\br{S\subst{i}U};j,V,W})\qquad\qquad }
\\
&C_{\br{S\subst{i}T}}\oplus
C_{\br{S\subst{i}U}}
\ar[uur]_{\rho_{\br{S};i,T,U}}
} $$
\end{enumerate}
A morphism $f\colon \{C_{\br{S}},\rho_{\br{S};i,T,U}\}\to
\{C'_{\br{S}},\rho'_{\br{S};i,T,U}\}$
consists of morphisms $f_{S}\colon C_{\br S}\to C'_{\br S}$ in
$\C$ for all $\br S$ such that $f_{\br S}$ is the identity
$\id_{0}$ when $S_{i}=\emptyset$ for any $i$, and the following
diagram commutes for all $\rho_{\br{S};i,T,U}$:
$$ \xymatrix@C+15pt{
C_{\br{S\subst{i}T}}\oplus
C_{\br{S\subst{i}U}}\ar[r]^-{\rho_{\br{S};i,T,U}}
\ar[d]_{f_{\br{S\subst{i}T}}\oplus f_{\br{S\subst{i}U}}}
&C_{\br{S}}\ar[d]^{f_{\br{{S}}}}\\
C'_{\br{S\subst{i}T}}\oplus
C'_{\br{S\subst{i}U}}\ar[r]_-{\rho'_{\br{S};i,T,U}}
&C'_{\br{S}}.
 } $$
Note that if any of the $\b n_i=\b0$ in the definition above, then
$\br{\b n}=*$, so $J\C\br{\b n}$ must be the terminal category
with one object and one morphism.
\end{definition}

The following theorem is \cite{RMA}, Theorem 6.1.

\begin{theorem}
The categories $J\C\br{\b n}$ support the structure of a
$\Gstar$-category.
\end{theorem}

The $\Gstar$-category structure is constructed as follows.
First, for a fixed string length $k$, so $\br{\b n}=(\b
n_1,\dots,\b n_k)$, $J\C\br{\b n}$ is functorial in morphisms of
$\F^k$, as follows.  Given maps $\alpha_i:{\b m_i}\to{\b n_i}$ of
based sets for $1\le i\le k$, we define
$$J\C\br{\alpha}:J\C\br{\b m}\to J\C\br{\b n}
$$
on objects by
$$J\C\br{\alpha}\{C_{\br{S}},\rho_{\br{S};i,T,U}\}
:=\{C_{\br{S}}^\alpha,\rho_{\br{S};i,T,U}^\alpha\},
$$
where
$$C_{\br{S}}^\alpha
=C_{(\alpha_1^{-1}S_1,\dots,\alpha_k^{-1}S_k)}
$$
and
$$\rho_{\br{S};i,T,U}^\alpha
=\rho_{\br{\alpha^{-1}S};i,\alpha^{-1}T,\alpha^{-1}U},
$$
and similarly on morphisms.  Note that since the $\alpha_i$ are
based maps, $\alpha_i^{-1}S_i$ is a subset of $\{1,\dots,m_i\}$
for all $i$.

Next, a permutation $\sigma\in\Sigma_k$ induces a functor
$$\sigma_{!}\colon J\C{(\b n_1,\dots,\b n_k)}
\to J\C{(\b n_{\sigma^{-1}(1)},\dots,\b n_{\sigma^{-1}(k)})},
$$
which is an isomorphism of categories, as follows:  The object
$\{C_{\br{S}},\rho_{\br{S};i,T,U}\}$ is sent to the object
$\{C^{\sigma}_{\br{S'}},\rho^{\sigma}_{\br{S'};i,T}\}$ where
$$
C^{\sigma}_{\br{S'}}
 = C_{\sigma\br{S'}},
\qquad
\rho^{\sigma}_{\br{S'};i,T,U}
 = \rho_{\sigma\br{S'};\sigma(i),T,U},
\qquad
\sigma\br{S'}=(S'_{\sigma(1)},\dotsc,S'_{\sigma(k)}),
$$
so if $S'_{i}=S_{\sigma^{-1}(i)}\subset \{1,\dots,
n_{\sigma^{-1}(i)}\}$, then $\sigma\br{S'}=\br{S}$. The morphism
$\{f_{\br{S}}\}$ is sent to the morphism
$\{f^{\sigma}_{\br{S'}}\}$ where $f^{\sigma}_{\br{S'}}=f_{\sigma
\br{S'}}$.  It is straightforward to verify that
$(\sigma\tau)_!=\sigma_!\tau_!$.

Finally, we have isomorphisms of categories
$$ e\colon J\C{(\b n_{1},\dots,\b n_{k})}\to
J\C{(\b n_{1},\dots,\b n_{k},\b1)}
$$
defined as follows: the object
$\{C_{\br{S}},\rho_{\br{S};i,T,U}\}$ is sent to the object
$\{C^{e}_{\br{S'}},\rho^{e}_{\br{S'};i,T,U}\}$, where
\begin{align}
C^{e}_{(S_{1},\dots,S_{k},\{1\})}
 &= C_{\br{S}},
&
\rho^{e}_{(S_{1},\dots,S_{k},\{1\});i,T,U}
 &= \rho_{\br{S};i,T,U} &{\text{ for } i<k+1,}
\notag\\
C^{e}_{(S_{1},\dots,S_{k},\emptyset)}&=0,
&
\rho^{e}_{(S_{1},\dots,S_{k},\emptyset);i,T,U}
 &= \id,&
\rho^{e}_{(S_{1},\dots,S_{k},\{1\});k+1,T,U}&=\id.
\notag
\end{align}

The morphism $\{f_{\br{S}}\}$ is sent to the morphism
$\{f^{e}_{\br{S'}}\}$ where
$$f^{e}_{(S_{1},\dots,S_{k},\{1\})}=f_{\br{S}}, \qquad
f^{e}_{(S_{1},\dots,S_{k},\emptyset)}=\id.
$$
This description of the components of the objects and morphisms is
complete since the only two subsets of $\{1\}$ are $\{1\}$ and
$\emptyset$. The inverse of this isomorphism is induced by
dropping the $\{1\}$ from $(k+1)$-tuples of the form
$(S_{1},\dotsc,S_{k},\{1\})$.  This describes image functors for a
generating set of morphisms of $\Gstar$, and since $J\C\br{\b
n}=*$ if any of the $\b n_i=\b0$, it is now easy to verify that we
do in fact get a $\Gstar$-category $J\C$.

We now begin the construction of the natural isomorphism $\Jhat
U\C\cong J\C$, and we proceed objectwise in $\Gstar$, so we need
to produce isomorphisms of categories
$$\Jhat U\C\br{\b m}\cong J\C\br{\b m}
$$
for each object $\br{\b m}=(\b m_1,\dots,\b m_k)$ of $\Gstar$. The
bulk of the construction is concerned with the bijection on
objects. Suppose given an object of $\Jhat U\C\br{\b m}$, that is,
of $\Multstar(E^*\br{\b m},U\C)$, say $F:E^*\br{\b m}\to U\C$,
with $\br{\b m}\ne*$. We need to produce an object of $J\C\br{\b
m}$. But the objects of $E^*\br{\b m}$ can be considered as
$k$-tuples $(S_1,\dots, S_k)$ where $S_i\subset\{1,\dots,m_i\}$,
so we get the part of an object of $J\C$ given by a system
$C_{\br{S}}$ by defining
$$C_{\br{S}}:=F\br{S}.
$$
We also need to produce the structure maps in the system, so
suppose given subsets $T$ and $U$ of $\{1,\dots,m_i\}$ with $T\cap
U=\emptyset$ and $T\cup U=S_i$.  We define the associated
structure map $\rho_{\br{S};i,T,U}$ to be the image under $F$ of
the $2$-morphism in $E^{m_i}$ given by
$$(T,U)\to T\cup U=S_i,
$$
together with the objects in the other slots in $\br{S}$.  Now the
coherence properties (1) and (2) follow from $F$ being a based
multifunctor, (3) follows from the commutative diagram
$$\xymatrix{
(T,U)\ar[r]\ar[d] &T\cup U\ar[d]^-{=}
\\(U,T)\ar[r]&T\cup U}
$$
in $E^{m_i}$, (4) follows from the commutative diagram
$$\xymatrix{
(T,U,V)\ar[r]\ar[d]&(T,U\cup V)\ar[d]
\\ (T\cup U,V)\ar[r]&T\cup U\cup V}
$$
in $E^{m_i}$, and (5) follows from bilinearity.

The reverse direction is the most significant part of the proof:
given an object $(C_{\br{S}},\rho_{\br{S};i,T,U})$ of $J\C\br{\b
m}$, we need to construct a multifunctor $F:\Estar\br{\b m}\to
U\C$.  The map is clear on objects: $F\br{S}:=C_{\br{S}}$.  Now
suppose given an $n$-morphism in $E^{m_i}$, say
$(T_1,\dots,T_n)\to S_i$, so $T_r\cap T_s=\emptyset$ unless $r=s$,
and $T_1\cup\dots\cup T_n=S_i$.  We need to construct the image
$n$-morphism in $U\C$ under our multifunctor $F$, which will be a
morphism in $\C$
$$C_{\br{S\subst{i}T_1}}\oplus\dots\oplus C_{\br{S\subst{i}T_n}}
\to C_{\br{S}}.
$$
We define this inductively, requiring the morphism to be $\id_0$
if $n=0$ and $\id_{C_{\br{S}}}$ if $n=1$.  For larger $n$'s, we
define the image $n$-morphism by induction to be the composite
$$\xymatrix{
C_{\br{S\subst{i}T_1}}\oplus\dots\oplus C_{\br{S\subst{i}T_{n-1}}}
\oplus C_{\br{S\subst{i}T_n}}\ar[r]
&C_{\br{S\subst{i}(S\setminus T_n)}}\oplus C_{\br{S\subst{i}T_n}}
\ar[rr]^-{\rho_{\br{S};i,S_i\setminus T_n,T_n}}
&&C_{\br{S}},}
$$
where the first map is given by induction on the first $n-1$
terms, and the second is the structure map given by the object of
$J\C\br{\b m}$.

We must verify that this definition actually gives a multifunctor
$F:E^*\br{\b m}\to U\C$, so we must show that it respects the
composition $\Gamma$ and the action of $\Sigma_n$ on the set
of $n$-morphisms.  By the definition of $E^*\br{\b m}$, this
reduces to checking multifunctoriality in each $\b m_i$
separately, and then bilinearity in each pair, using the based
concepts in both cases.  For notational convenience, we assume
without loss of generality that the list $\br{\b m}=(\b
m_1,\dots,\b m_k)$ has length 1 for the first part of this check,
so $\br{\b m}=\b m$ is an object of $\F$.

Our first step is the following lemma.

\begin{lemma}\label{cut} Let $T_1,\dots,T_i,U_1,\dots,U_j$ be a
collection of mutually disjoint subsets of $\{1,\dots,m\}$.  Let
$T=T_1\cup\dots\cup T_i$ and $U=U_1\cup\dots\cup U_j$.  Then the
morphism induced by the $i+j$-morphism
$$(T_1,\dots,T_i,U_1,\dots,U_j)\to T\cup U
$$
in $E^m$ factors through maps induced by morphisms in $E^m$ as
indicated in the following:
$$C_{T_1}\oplus\dots\oplus C_{T_i}\oplus C_{U_1}\oplus C_{U_j} \to
C_T\oplus C_U\to C_{T\cup U}.
$$
\end{lemma}

\begin{proof} We induct on $j$, and the claim is trivially true if
$j=0$ or $j=1$.  For the general case, we examine the following
diagram, in which all arrows are induced from morphisms in $E^m$:
$$\xymatrix@C-100pt{
&&&&&&{C_{T_1}\oplus\dots\oplus C_{T_i}\oplus C_{U_1}\oplus\dots\oplus
C_{U_{j-1}}\oplus C_{U_j}}\ar@/_2pc/[ddlll]\ar[dlll]\ar[dllllll]
\\C_{T\cup U\setminus U_j}\oplus C_{U_j}\ar[d]
&&&C_T\oplus C_{U\setminus U_j}\oplus C_{U_j}\ar[lll]\ar[d]
\\C_{T\cup U}.
&&&C_T\oplus C_U\ar[lll]
}
$$
The top triangle commutes by induction, the left triangle commutes
by definition of the induced maps, and the square is the
associativity condition for the structure maps of an object of
$J\C\br{\b m}$, property (4). The clockwise composite is the
definition of the induced map, and the conclusion follows.
\end{proof}

We can now show that our construction respects the
composition $\Gamma$.  Suppose we have mutually disjoint
objects $S_{11},\dots,S_{1r_1},\dots,S_{n1},\dots,S_{nr_n}$ of
$E^m$, so we have $S_{ij}$'s for $1\le i\le n$ and $1\le j\le
r_i$.  We write $S_i=S_{i1}\cup\dots\cup S_{ir_i}$ and
$S=S_1\cup\dots\cup S_n$.  To show our construction preserves
composition, we must show that the composite of induced maps
$$C_{S_{11}}\oplus\dots\oplus C_{S_{1r_1}}\oplus\dots\oplus
C_{S_{n1}}\oplus\dots\oplus C_{S_{nr_n}}\to
C_{S_1}\oplus\dots\oplus C_{S_n}\to C_S
$$
is the induced map.  We examine the following diagram, in which
all maps are induced from morphisms in $E^m$, and proceed by
induction on $n$:
$$\xymatrix@C-122pt{
&&&&&&C_{S_{11}}\oplus\dots\oplus C_{S_{1r_1}}\oplus\dots\oplus
C_{S_{n1}}\oplus\dots\oplus
C_{S_{nr_n}}\ar@/_1.5pc/[ddlll]\ar[dlll]\ar[dllllll]\ar[dd]
\\C_{S\setminus S_{nr_n}}\oplus C_{S_{nr_n}}\ar[d]
&&&C_{S\setminus S_n}\oplus C_{S_n\setminus S_{nr_n}}\oplus C_{S_{nr_n}}
\ar[lll]\ar[d]
\\C_S.
&&&C_{S\setminus S_n}\oplus C_{S_n}\ar[lll]
&&&C_{S_1}\oplus\dots\oplus C_{S_n}\ar[lll]
}
$$
Reading from left to right, the triangles out of the top left
entry commute by induction, by definition, and by Lemma \ref{cut},
while the square is another instance of the associativity
property.  Now the clockwise composite defines the total induced
map, while the counterclockwise composite is the given one.  The
conclusion follows.

In order to check that our construction respects the permutation
actions, it suffices to check preservation of transposition of
adjacent letters, since these generate $\Sigma_n$.  Suppose given
mutually disjoint subsets $T_1,\dots,T_i,U_1,\dots,U_j$ of
$\{1,\dots,m\}$, and we write as before $T=T_1\cup\dots\cup T_i$
and $U=U_1\cup\dots\cup U_j$.  We need to show that
$$\xymatrix{
C_{T_1}\oplus\dots\oplus C_{T_{i-1}}\oplus C_{T_i}\oplus
C_{U_1}\oplus\dots\oplus C_{U_j}\ar[dd]_-{\cong}\ar[dr]
\\&C_{T\cup U}
\\C_{T_1}\oplus\dots\oplus C_{T_i}\oplus C_{T_{i-1}}\oplus
C_{U_1}\oplus\dots\oplus C_{U_j}\ar[ur]}
$$
commutes.  This follows from the case $j=0$, however, by the
following diagram, in which all maps are induced from morphisms in
$E^m$:
$$\xymatrix@C-20pt{
C_{T_1}\oplus\dots\oplus C_{T_{i-1}}\oplus C_{T_i}\oplus
C_{U_1}\oplus\dots\oplus C_{U_j}\ar[dd]_-{\cong}\ar[dr]\ar[drrrr]
\\&C_T\oplus C_U\ar[rrr]
&&&C_{T\cup U}.
\\C_{T_1}\oplus\dots\oplus C_{T_i}\oplus C_{T_{i-1}}\oplus
C_{U_1}\oplus\dots\oplus C_{U_j}\ar[ur]\ar[urrrr]}
$$
The left triangle follows from the case $j=0$, and the other two
triangles are instances of Lemma \ref{cut}.  The case $j=0$
follows by examining the diagram
$$\xymatrix@C-20pt{
C_{T_1}\oplus\dots\oplus C_{T_{i-1}}\oplus
C_{T_i}\ar[dd]_-{\cong}\ar[dr]\ar[drrrr]
\\&C_{T\setminus(T_{i-1}\cup T_i)}\oplus C_{T_{i-1}\cup T_i}\ar[rrr]
&&&C_{T}.
\\C_{T_1}\oplus\dots\oplus C_{T_i}\oplus C_{T_{i-1}}\ar[ur]\ar[urrrr]
}$$ The left triangle follows from property (3), the transposition
axiom for the structure maps of objects in $J\C\br{\b m}$, and the
other two triangles are further instances of Lemma \ref{cut}.  The
construction is therefore multifunctorial.

We now return to the full generality of objects of $E^*\br{\b m}$,
and must verify that our construction is based bilinear.
Bilinearity follows from property (5) of an object of $J\C\br{\b
m}$ by the same argument used in the proof of Theorem \ref{emb} to
prove that bilinearity follows from a pentagon diagram, using
Figure \ref{Fi:1}.  Finally, basedness follows from properties (1)
and (2) of an object of $J\C\br{\b m}$, requiring that
$C_{\br{S}}=0$ whenever any $S_i=\emptyset$, and that the
structure map be the identity whenever any $S_i$, $T$, or $U$ is
empty.  This completes the verification that our construction
produces a multifunctor $E^*\br{\b m}\to U\C$ from an object of
$J\C$.

We must show that these correspondences are inverse to each other.
Given $F:E^*\br{\b m}\to U\C$, we produce the
object$(C_{\br{S}},\rho)$ of $J\C\br{\b m}$ with the system of
objects of $\C$ given by
$$C_{\br{S}}:=F\br{S}.
$$
From this we redefine a multifunctor from $E^*\br{\b m}$ to $U\C$,
where given an $n$-morphism $(T_1,\dots,T_n)\to S_i$ in $E^{m_i}$,
we use the inductive definition given by the composite
$$\xymatrix{
C_{\br{S\subst{i}T_1}}\oplus\dots\oplus C_{\br{S\subst{i}T_{n-1}}}
\oplus C_{\br{S\subst{i}T_n}}\ar[r]
&C_{\br{S\subst{i}(S\setminus T_n)}}\oplus C_{\br{S\subst{i}T_n}}
\ar[rr]^-{\rho_{\br{S};i,S\setminus T_n,T_n}}
&&C_{\br{S}}}
$$
to define the image $n$-morphism in $U\C$.  However, this must
coincide with the image $n$-morphism given by our original $F$ by
induction and the commutativity in $E^{m_i}$ of the diagram
$$\xymatrix{
(T_1,\dots,T_{n-1},T_n)\ar[r]\ar[dr]
&(S_i\setminus T_n,T_n)\ar[d]
\\&S_i.}
$$

Conversely, suppose given an object $(C_{\br{S}},\rho)$ of $J\C$.
Then we define the corresponding multifunctor $F$ by setting
$F\br{S}:=C_{\br{S}}$ and using induction to define the
correspondence on $n$-morphisms.  But now taking that multifunctor
and recovering the corresponding object of $J\C$ takes us back to
the system of objects $C_{\br{S}}$, and the structure maps $\rho$
are all induced by maps given by $F$ from the original object of
$J\C$.  We therefore recover the original object, and we have
shown that our correspondences give inverse bijections between the
objects of $J\C$ and of $\Mult_*(E^*,U\C)$.

In order to show that we also get inverse bijections on morphisms,
and therefore isomorphisms of categories, we just note that the
morphisms in both $J\C$ and $\Mult_*(E^*,U\C)$ are given by
natural transformations, and that our constructions in each
direction give inverse correspondences of natural transformations.
Preservation of composition and naturality in $\G$ are easy
exercises left to the reader.  We have finished showing that our
construction extends that of \cite{RMA}, and therefore the proof
of Theorem \ref{extn}.

\end{document}